\documentclass{article}
\usepackage[utf8]{inputenc}
\usepackage{pdfpages}
\usepackage{amsthm}
\usepackage{amsmath}
\usepackage[hidelinks]{hyperref}
\usepackage[nameinlink]{cleveref}
\usepackage{url}
\usepackage{thmtools}
\usepackage{amssymb}
\usepackage{float}
\usepackage{pdflscape}
\usepackage{dsfont}
\usepackage{natbib}
\usepackage{placeins}

\usepackage[labelformat=simple]{subcaption}

\usepackage{tikz}
\usetikzlibrary{patterns}
\usetikzlibrary{matrix}
\usetikzlibrary{graphs,quotes}\usepackage{pgfplots}
\pgfplotsset{compat=1.15,
	every axis/.append style={
		axis lines=center,
		xlabel style={anchor=south west},
		ylabel style={anchor=south west},
		zlabel style={anchor=south west},
		tick align=outside,}
}
\usepgfplotslibrary{patchplots}
\pgfdeclareverticalshading{darkGray}{100bp}
{color(0bp)=(black!45); color(32bp)=(black!45); color(38bp)=(black!40);color(45bp)=(black!30); color(100bp)=(black!25)}

\pgfdeclareverticalshading{darkGray2}{100bp}
{color(0bp)=(black!45); color(40bp)=(black!45); color(45bp)=(black!40);color(50bp)=(black!30); color(100bp)=(black!25)}

\pgfdeclareverticalshading{midGray}{100bp}
{color(0bp)=(black!30); color(30bp)=(black!35);color(35bp)=(black!30);color(40bp)=(black!25); color(47bp)=(black!15); color(100bp)=(black!10)}

\pgfdeclareverticalshading{lightGray}{100bp}
{color(0bp)=(black!20); color(28bp)=(black!15);color(35bp)=(black!15); color(40bp)=(black!10); color(45bp)=(black!5); color(100bp)=(black!5)}

\pgfdeclareverticalshading{lightGray2}{100bp}
{color(0bp)=(black!20); color(35bp)=(black!15);color(40bp)=(black!15); color(45bp)=(black!10); color(50bp)=(black!5); color(100bp)=(white)}

\usepackage{cancel}

\usepackage[ruled,vlined,linesnumbered]{algorithm2e}
\SetKwProg{Fn}{Function}{}{}
\DontPrintSemicolon
\SetKwFor{For}{for}{}{}
\SetKwFor{Forall}{for all}{}{}
\SetKwFor{While}{while}{}{}

\SetCommentSty{mycommfont}

\usepackage{rotating} \usepackage{booktabs} \usepackage{mathtools}
\usepackage{cleveref}
\crefname{algocf}{alg.}{algs.}
\Crefname{algocf}{Algorithm}{Algorithms}

\newtheorem{defi}{Definition}
\newtheorem{theorem}[defi]{Theorem}
\newtheorem{lem}[defi]{Lemma}

\newtheorem{rem}[defi]{Remark}

\newtheorem{example}[defi]{Example}

\newcommand {\N}{\mathds{N}}

\newcommand {\R}{\mathds{R}}

\DeclareMathOperator{\update}{update}

\newcommand{\g}{\gamma}

\newcommand {\E}{E}     

\makeatletter
\addtotheorempostfoothook{\global\everypar{{\setbox\z@\lastbox}\global\everypar{}}}
\makeatother

\newcommand{\facloc}{v} \newcommand{\ivPaths}{\mathcal{P}_i(\facloc)} \newcommand{\allPaths}{\mathcal{P}_V(\facloc)} \newcommand{\OnevPaths}{\mathcal{P}_1(\facloc)}
\newcommand{\nvPaths}{\mathcal{P}_n(\facloc)}
\newcommand{\selectedPaths}{P_V(\facloc)} \newcommand{\selectedPathsprime}{P_V(\facloc')} \newcommand{\selectedPathsC}{P_V(\facloc^*)} \newcommand{\selectedPathi}{P_i(\facloc)} \newcommand{\selectedPathOne}{P_1(\facloc)}

\newcommand{\selectedPathN}{P_n(\facloc)}
 \newcommand{\selectedPathiprime}{P'_i(\facloc)}

\newcommand{\Path}{P} \newcommand{\AltPath}{Q} \newcommand{\C}{\mathcal{C}} \newcommand{\X}{\mathcal{X}} \newcommand{\cat}{\eta} 

\newcommand{\sol}{z=(\facloc,\selectedPaths)} \newcommand{\solprime}{z'=(\facloc',\selectedPathsprime)} \newcommand{\solC}{z^*=(\facloc^*,\selectedPathsC)} 

\newcommand{\SCon}{SC} \newcommand{\Con}{C} \newcommand{\VSCon}{SC(\facloc)} \newcommand{\VCon}{C(\facloc)} 

\DeclareMathOperator{\up}{up}
\newcommand{\parS}{(\Vc,\Ec)}
\newcommand{\Aup}{A^{\up}}
\newcommand{\aup}{a^{\up}}
\newcommand{\Vc}{V_c}
\newcommand{\Ec}{E_c}
\newcommand{\VP}{V_P}
\newcommand{\EP}{E_P}
\newcommand{\M}{\mathcal{M}}
\newcommand{\At}{A_t} \newcommand{\A}{\mathcal{A}} %

\newcommand{\highlight}[1]{\textcolor{teal}{#1}}

\begin{document}

\title{Consistent Path Selection for Bi-objective Median Location Problems on Graphs}
\author{Renée Lamsfuß \and Kathrin Klamroth \and Michael Stiglmayr \and Julia~Sudhoff~Santos}

\date{\small University of Wuppertal, School of Mathematics and Natural Sciences, Gaußstr.~20, 42119 Wuppertal, Germany\\ \texttt{\{mlamsfuss,klamroth,stiglmayr,sudhoff\}@uni-wuppertal.de}}

\maketitle

\begin{abstract}
We consider single-facility median location problems on graphs where two conflicting cost values are associated with the edges.
As an example, {suppose that} a decision maker wants to locate one new facility, e.\,g., a pizza delivery place, which uses bicycles for delivery. The two objective functions could then be the total traveling time and the total number of left turns, as the latter are very risky. Path choices then depend on the preferences of the decision maker, and in general, there may not exist a unique optimal path between a customer and a new facility location. 

In this paper, we consider the location decision and the routing decision in a coupled problem, i.\,e., we search for an optimal location and for consistent delivery paths simultaneously.
We introduce the concept of consistent paths, where we assume that the choice of a path from the facility to a demand node implies certain preferences. All paths of a solution are consistent if the preferences of all paths do not contradict each other. We present an algorithm that computes {a minimum complete set of} efficient solutions with consistent path choices and illustrate the results at example instances in the city of Wuppertal in Germany. 

\end{abstract}

\section{Introduction}
We consider bi-objective median location problems on graphs, i.\,e., we want to locate one new facility on a simple undirected graph such that the sum of bi-objective edge costs from the new facility to a finite set of customers (located at the nodes of the graph) is minimized. 
As an example, suppose that a decision maker wants to open a new pizza delivery service from which pizza is delivered to customers by bicycle. {Each edge is associated with two cost coefficients: The first measures the travel time, and the second counts the number of left turns and is hence a measure for the safety of the considered route. (Note that the latter may require the introduction of artificial nodes and edges at intersections.)} In this situation, we aim to find a location for the pizza shop \emph{together} with delivery routes {to all customers} such that both the total travel time and the number of left turns are minimized. {
Note that,} in general, there does not exist one unique safest and shortest path from the new facility to a customer but a set of efficient compromise solutions. 

When safety is a major concern, it can also be modeled by means of an ordinal objective function {that assigns safety categories to all edges of the graph (in addition to their respective lengths)}. If this ordinal objective function considers two categories, the problem can {also} be reformulated as a bi-objective problem, {i.e., as a location problem with bi-objective edge costs}. See \cite{Klamroth2023Ordinal} for an in-depth discussion of ordinal optimization problems in general.  Ordinal location problems with two categories are used for illustration throughout this paper. Note that there exist different concepts of optimality for ordinal costs, which do not all comply with Bellmann's principle of optimality, and hence do not lead to meaningful paths in the context of the location problems. Such concepts were considered, e.\,g., in \cite{Schaefer2020Shortest}. For a discussion of shortest path problems and facility location problems we also refer to \cite{froehlich_diss}.

The literature on bi- and multi-objective location problems often assumes vector-valued demands, where each scenario defines one objective function; see, e.\,g., \cite{Nickel2019Location}. Multi-objective edge lengths are considered in \cite{Nickel05Location}. However, they consider these edge lengths as independent, i.\,e., different paths can be chosen for each objective function. Multi-objective demands and multi-objective edge lengths (still allowing independent route choices per objective) are combined, e.\,g., in \cite{Colebrook2007polynomial} and \cite{Colebrook2007Undesirable}. 
To the best of our knowledge, the only work that considers bi-objective shortest path problems in combination with median location problems is \cite{Skriver2004Bicriteria}. In this work, a new facility is to be chosen together with unique routes to all customers, i.\,e., the routes must be the same for both objective functions. We draw connections to \cite{Skriver2004Bicriteria} throughout this paper and discuss similarities as well as differences. Bi-objective shortest path problems with centrality measures are discussed in \cite{SIMONDEBLAS2026103575}.

The main contribution of this paper is the formulation, analysis, and algorithmic solution of bi-objective, single-facility median location problems on graphs with consistent route planning. We consider the location decision \emph{and} the routing decision as a coupled problem and hence search for the location \emph{and} the \emph{consistent} delivery routes simultaneously. 

In this context, we consider a \emph{central choice model} (CC), i.\,e., a central decision maker selects paths to all customers according to his/her preferences. The choices of the decision maker are said to be \emph{consistent} if the underlying preference relation stays the same for the paths to all customers.
This model represents, for example, the optimal location of a restaurant that delivers food to customers by bicycle. Let there exist, for example, a route to the first customer, which takes 20 minutes and 3 left turns, and another route which takes 15 minutes and 8 left turns. If the decision maker decides to use the first route, this means that he/she prefers 5 left turns less over a reduction of the travel time by 5 minutes. Then, we assume that this preference also holds for all decisions regarding the routes to the other customers. 
Note that in other situations, e.\,g., when individual customers visit the new facility, \emph{distributed choice models} (DC) may be more appropriate. This case is not considered in this paper.

The remainder of the paper is organized as follows. In Section~\ref{sec:olp}, we introduce bi-objective location {and path planning} problems, define consistency and show that we can restrict the set of feasible locations to the node set of the graph. Moreover, we introduce a specific bi-objective location {and path planning} problem, the ordinal location {and path planning} problem with two categories, which will be used for illustration throughout this paper. We show in Section~\ref{sec:specific} that for specific graphs, where the path selection is unique, the unique optimal solution can be computed easily {(even in the multi-objective case)}. In Section~\ref{sec:generalgraphs} solution algorithms for general graphs are introduced. Numerical tests on small instances selected from the street network of Wuppertal are presented in Section~\ref{sec:numeric}. We conclude this paper in Section~\ref{sec:concl} with a short summary and outlook.

\section{Bi-objective Location Models and Consistent Paths}\label{sec:olp}

Let $G=(V,E)$ be a simple, connected, and undirected graph with node set~$V$ with $|V|=n$ and edge set~$E\subseteq \{[u,v]\colon u,v\in V\}$ with $|E|=m$. We assume that all nodes $v_i\in V$ represent customers with non-negative demands $w_i\geq0$, where at least one $w_j$ exists with $w_j>0, j\in\{1,\dots,n\}$, to exclude the trivial case. We consider discrete location problems where one new facility is to be located at one of the nodes in $V$ (i.\,e., at one of the customer locations). A path $\Path$ from \(v_i\) to \(v_j\) is represented by the sequence of its nodes $P=(v_i=v_{k_1},v_{k_2},\ldots,v_{k_{{\ell}}}=v_j)$, where $[v_{k_i},v_{k_{i+1}}]\in E$ for all $i=1,\dots,{\ell}$. All edges $e=[v_i,v_j]\in E$ have bi-objective, component-wise non-negative costs $c(e)\in\R^2_\geq$, where $\R^p_{\geq}=\{y\in\R^p:y_i\geq0\;\text{for all }i\in\{1,\ldots,p\},y\neq0\}$ is defined with respect to the component-wise order (for $p\in\N$, $p\geq 2$).

With $\allPaths=(\OnevPaths,\dots,\nvPaths)$ we denote the tuple of \emph{all} facility-to-customer paths from a facility location $\facloc$ to the other nodes $v_i\in V$, $i=1,\dots,n$. Let $\sol$ be a tuple with $v\in V$ denoting a facility location in $G$ and let $\selectedPaths=\{\selectedPathOne,\dots,\selectedPathN\}$ be the set of corresponding facility-to-customer paths, i.\,e., let $\selectedPathi$ be a \emph{selected} $(\facloc,v_i)$-path in $G$. We denote the set of all such tuples with $\X$. 
We define the objective value $c(\selectedPaths)\in\R^2_\geq$ of a feasible tuple ${\sol\in\X}$ and  the costs $c(\Path)\in\R^2_\geq$ of a path~$\Path$  as
\[
c(\selectedPaths)\coloneqq \sum_{i=1}^n w_i \,c(\selectedPathi), \quad
\text{where} \; c(\Path)\coloneqq \sum_{e\in\Path}c(e).
\]

For algorithmic purposes, we formulate the \emph{bi-objective location and path planning problem} \eqref{BLP} with respect to a preference matrix $A$ as
\begin{equation}\label{BLP}\tag{BLPP}
        \begin{aligned}
            \min\enspace &f(z)=A\cdot c(\selectedPaths)\\
            s.t.\enspace &\sol\in \X,
        \end{aligned}
\end{equation}
where the \textit{preference matrix} $A\in\A_2$ and $\A_p=\{A\in\R^{p\times p}_\geq:a_{ii}=1,i=1,\dots p\}$ for $p\in\N$. Preference matrices will be primarily used within solution algorithms in order to adaptively and automatically update partial preferences, and to ensure a consistent path selection within a solution $\sol$.
For general problems, we initialize $A=I$ with the identity matrix.
We discuss the role of the preference matrix in detail in \Cref{sec:generalgraphs}. {We also refer to \citet{Kaddani2017Weighted} for modeling partial preferences in multi-objective optimization, in general.}

Given two solutions $\sol$ and $\solprime\in\X$, $z$ dominates $z'$ if $f(z)\leq f(z')$ with respect to the component-wise order.
If there exists no $\bar{z}\in\X$ such that $f(\bar{z})\leq f(z)$, we say $z$ is an \emph{efficient} or \emph{Pareto optimal} solution, and $f(z)$ is a \emph{non-dominated} outcome vector. If a solution $z$ can be obtained as optimal solution of a weighted sum scalarization with strictly positive weights, the solution is called \emph{supported} \citep[e.\,g.,][]{koenen25supportedness}. A solution $z$ is called \emph{extreme supported} if $f(z)$ is located on an extreme point of the convex hull of the feasible images $\mathcal{Y}=f(\X)$. In this context, we distinguish supported solutions that solve \eqref{BLP} for a prefixed facility location $v\in V$ and refer to them as \emph{locally supported} from those \emph{globally supported} solutions that solve \eqref{BLP} among all possible facility locations \citep{Skriver2004Bicriteria}. 

\subsection{Consistent Paths}
If the central decision maker has decided on a certain path to a demand node, then this decision implies certain preferences. If, for example, a path with cost $(10,1)^\top$ is preferred over a path with cost $(1,4)^\top$, then this implies that three units in the first objective are preferred over one unit in the second objective. We now assume that the central decision maker makes consistent path choices to all other demand nodes, such that all paths comply with this preference relation. This formally leads to the following definition of consistency.

\begin{defi}\label{def:consistency}
Let $\sol\in\X$ be a feasible solution of the bi-objective location and path planning problem \eqref{BLP}, where $\selectedPaths=\{\selectedPathOne,\dots,\selectedPathN\}$, and let $\ivPaths$ denote the set of all $(\facloc,v_i)$-paths in the graph $G=(V,E)$ for all $v_i\in V,i\in\{1,\dots,n\}$ and $v\in V$.

\begin{description}
    \item[$(\VCon)$\label{consistent(v)}] A solution $\sol$ and its outcome vector $f(z)$ are called \emph{vertex consistent w.\,r.\,t.~vertex $\facloc$} if there exists a  weight vector $\lambda\in\R^2_\geq$ 
    such that for all $i=1,\dots,n$ and for all $\selectedPathiprime\in\ivPaths$ it holds that $\lambda^\top\cdot A\cdot c(\selectedPathi)\leq \lambda^\top\cdot A\cdot c(\selectedPathiprime)$ with $\selectedPathi\in\selectedPaths$.
    \item[$(\VSCon)$\label{stronglyconsistent(v)}] A solution $\sol$ and its outcome vector $f(z)$ are called \emph{strongly vertex consistent w.\,r.\,t.~vertex $\facloc$} if there exists a weight vector $\lambda\in\R^2_\geq$ 
    such that for all $i=1,\dots,n$ and for all $\selectedPathiprime\in\ivPaths$ with $c(\selectedPathi)\neq c(\selectedPathiprime)$ it holds that $\lambda^\top\cdot A\cdot c(\selectedPathi)<\lambda^\top\cdot A\cdot c(\selectedPathiprime)$ with $\selectedPathi\in\selectedPaths$.
    \item[$(\Con)$\label{consistent}]
    A solution $\solC$ and its outcome vector $f(z^*)$ are called \emph{consistent} if there exists a weight vector $\bar{\lambda}\in\R^2_\geq$ such that $\bar{\lambda}^\top\cdot f(z^*)\leq \bar{\lambda}^\top\cdot f(z)\enspace\forall z\in\X$. 
    \item[$(\SCon)$\label{stronglyconsistent}] A solution $\solC$ and its outcome vector $f(z^*)$ are called \emph{strongly consistent} if there exists a weight vector $\bar{\lambda}\in\R^2_\geq$ such that $\bar{\lambda}^\top\cdot f(z^*)< \bar{\lambda}^\top\cdot f(z)\enspace\forall z\in\X$ with $c(z^*)\neq c(z)$.
\end{description}
Solutions that are not (vertex) consistent are referred to as (vertex) inconsistent.
\end{defi}

Note that every solution that is strongly vertex consistent is also vertex consistent and that the set of strongly consistent solutions is a subset of the set of consistent solutions.
\begin{lem}\label{lem:consistency}
    Every (strongly) consistent solution $\sol$ is (strongly) vertex consistent w.\,r.\,t.\ vertex $\facloc$.
\end{lem}
\begin{proof}
    We show the result for a consistent solution $\sol$. We assume by contradiction that $\sol$ with $\selectedPaths=\{\selectedPathOne,\dots,\selectedPathN\}$ is not vertex consistent, i.\,e.\ for all ${\lambda\in\R^2_\geq}$  there exists an index $i\in\{1,\dots,n\}$ and  $\selectedPathiprime\in\ivPaths$ such that  $\lambda^\top\cdot A\, c(\selectedPathi)>\lambda^\top\cdot A\, c(\selectedPathiprime)$. 
    This implies that $\lambda^\top \cdot f(z)>\lambda^\top\cdot f(z')$ with $z'=(v,\{P_1(v),\dots,P_{i-1}(v),\selectedPathiprime,P_{i+1}(v),\dots,P_n(v)\})$, which is a contradiction to the consistency of $\sol$.
    Analogously, one can show the result for a strongly consistent solution.
\end{proof}

\begin{rem}
    From Definition~\ref{def:consistency}, it immediately follows that all consistent solutions are globally supported for \eqref{BLP} and all strongly consistent solutions are globally extreme supported. Moreover, it can be shown that all (strongly) vertex consistent solutions are locally (extreme) supported.
\end{rem}
To illustrate the concept of consistent paths, we consider a specific type of bi-objective location problems, namely ordinal location problems with two categories.

\subsection{Ordinal Location and Path Planning Problems \eqref{top1}}
A special case of bi-objective edge costs is obtained if, in addition to the edge length, each edge is assigned to one of two ordinal categories modeling, for example, the safety of the respective  edge. In this case, we assume that all edges $e=[v_i,v_j]\in E$ have a positive length $l(e)> 0$. 
Additionally, we associate each edge with one of two ordered categories; let $\C=\{\cat_1,\cat_2\}$ be the set of categories. The categories are ordered such that $\cat_1$ is preferred over $\cat_2$. We write $\cat_1\prec\cat_2$ and define the ordinal function $o(e)=\cat$, where $e\in E$ and $\cat\in\C$. Then, the cost vector for the ordinal problem is the \textit{counting vector} defined by
\[ c_q(e)\coloneqq \begin{cases}
        l(e), &\text{if } o(e)=\cat_q \\
        0, &\text{else}
    \end{cases}\qquad  \text{ for all } e\in E \text{ and  } q=1,2.\]

\begin{defi}[Ordinal Location and Path Planning Problem]
Let $c:\X\rightarrow\R_{\geq}^2$ be the counting vector of a feasible tuple. The ordinal location and path planning problem can be equivalently reformulated by means of a bi-objective optimization problem as shown in \cite{Klamroth2023Ordinal}.
In this paper, we {always} consider the \emph{transformed ordinal location and path planning problem \eqref{top1}}:
    \begin{equation}\label{top1}\tag{OLPP}
        \begin{aligned}
            \min\enspace &f(z)= \At \cdot c(\selectedPaths)\\
            s.t.\enspace &\sol\in \X,
        \end{aligned}
    \end{equation}
where $\At\in\A_2$ is defined as
\begin{equation*}\label{A}
\At=\begin{pmatrix}
1 & 1\\0 & 1 
\end{pmatrix}.
\end{equation*} 
\end{defi}
Note that Problem \eqref{top1} is a special case of Problem \eqref{BLP} with a specific choice of the preference matrix $A$ and the edge costs $c$. The preference matrix $\At$ expresses that one unit in category $\cat_1$ is always preferred over one unit in category $\cat_2$.

\begin{example}\label{ex:julia}
We use an instance of \eqref{top1} to illustrate the fact that there may exist solutions that are vertex consistent \nameref{consistent(v)}, but not strongly vertex consistent \nameref{stronglyconsistent(v)}.
    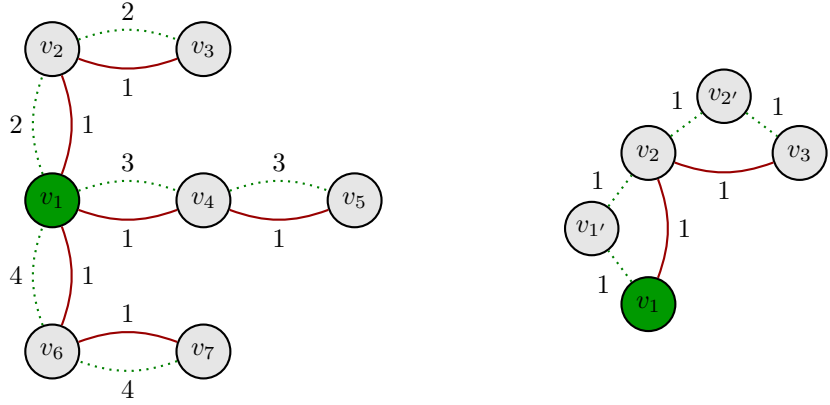
\begin{figure}
    \begin{minipage}[c]{0.6\textwidth}
        \centering
        \begin{tikzpicture}

\node[draw,circle,thick,fill=black!40!green] (1) at (0,0)[] {$v_1$};
\node[draw,circle,thick,fill=black!10] (2) at (0,2) {$v_2$};
\node[draw,circle,thick,fill=black!10] (3) at (2,2) {$v_3$};
\node[draw,circle,thick,fill=black!10] (4) at (2,0) {$v_4$};
\node[draw,circle,thick,fill=black!10] (5) at (4,0) {$v_5$};
\node[draw,circle,thick,fill=black!10] (6) at (0,-2) {$v_6$};
\node[draw,circle,thick,fill=black!10] (7) at (2,-2) {$v_7$};

\graph {
(1) --[bend right=20,thick,draw=black!40!red,swap,"1"] (2);
(1) --[bend left=20,dotted,draw=black!50!green,thick,"2"] (2);
(2) --[bend right=20,thick,draw=black!40!red,swap,"1"] (3);
(2) --[bend left=20,dotted,draw=black!50!green,thick,"2"] (3);
(6) --[bend right=20,thick,draw=black!40!red,swap,"1"] (1);
(6) --[bend left=20,dotted,draw=black!50!green,thick,"4"] (1);
(7) --[bend right=20,thick,draw=black!40!red,swap,"1"] (6);
(7) --[bend left=20,dotted,draw=black!50!green,thick,"4"] (6);
(1) --[bend right=20,thick,draw=black!40!red,swap,"1"] (4);
(1) --[bend left=20,dotted,draw=black!50!green,thick,"3"] (4);
(4) --[bend right=20,thick,draw=black!40!red,swap,"1"] (5);
(4) --[bend left=20,dotted,draw=black!50!green,thick,"3"] (5);

};

\end{tikzpicture}

     \end{minipage}
    \hfill
    \begin{minipage}[c]{0.3\textwidth}
        \centering
        \begin{tikzpicture}

\node[draw,circle,thick,fill=black!40!green] (1) at (0,0)[] {$v_1$};
\node[draw,circle,thick,fill=black!10] (2) at (0,2) {$v_2$};
\node[draw,circle,thick,fill=black!10] (3) at (2,2) {$v_3$};
\node[draw,circle,thick,fill=black!10, inner sep=2.5pt] (k) at (-0.75,1) {$v_{1'}$};
\node[draw,circle,thick,fill=black!10, inner sep=2.5pt] (j) at (1,2.75) {$v_{2'}$};

\graph {
(1) --[bend right=20,thick,draw=black!40!red,swap,"1"] (2);
(1) --[dotted,draw=black!50!green,thick,"1"] (k);
(k) --[dotted,draw=black!50!green,thick,"1",pos=0.25] (2);
(2) --[bend right=20,thick,draw=black!40!red,swap,"1"] (3);
(2) --[dotted,draw=black!50!green,thick,"1",pos=0.75] (j);
(j) --[dotted,draw=black!50!green,thick,"1"] (3);
};

\end{tikzpicture}
     \end{minipage}
    \caption{Left: Ordinal graph $G=(V,E)$ with the facility located at node $v_1$, categories safe ($\eta_1$, green dotted) and unsafe ($\eta_2$, red solid), edge lengths as depicted and {demands} $w_i=1$ for all nodes $v_i\in V,i\in 1,\dots,7$. Right: Subgraph of \(G\) illustrating the transformation to an equivalent simple graph}\label{fig:bsp_julia_excerpt_vollst}\label{fig:bsp_julia}
    \end{figure}

    \begin{figure}[ht!]
        \centering
        \begin{tikzpicture}

\node[draw,circle,thick,fill=black!40!green] (1) at (0,0)[] {$v_1$};
\node[draw,circle,thick,fill=black!10] (2) at (0,2) {$v_2$};
\node[draw,circle,thick,fill=black!10] (3) at (2,2) {$v_3$};
\node[draw,circle,thick,fill=black!10] (4) at (2,0) {$v_4$};
\node[draw,circle,thick,fill=black!10] (5) at (4,0) {$v_5$};
\node[draw,circle,thick,fill=black!10] (6) at (0,-2) {$v_6$};
\node[draw,circle,thick,fill=black!10] (7) at (2,-2) {$v_7$};

\graph {
(1) --[bend right=20,thick,draw=black!40!red,swap,"1"] (2);
(1) --[bend left=20,dotted,draw=black!50!green,thick,"2"] (2);
(2) --[bend right=20,thick,draw=black!40!red,"1"] (3);
(6) --[bend right=20,thick,draw=black!40!red,"1"] (1);
(7) --[bend right=20,thick,draw=black!40!red,"1"] (6);
(1) --[bend right=20,thick,draw=black!40!red,"1"] (4);
(4) --[bend right=20,thick,draw=black!40!red,"1"] (5);

};

\end{tikzpicture}
         \caption{Vertex consistent, but not strongly vertex consistent solution for the graph illustrated in \Cref{fig:bsp_julia}. The path $P_3(v_1)=(v_1,v_2,v_3)$ uses both red solid edges, while the path $P_2(v_1)=(v_1,v_2)=(v_1,v_{1'},v_2)$ uses the green dotted edge with edge length $2$.}
        \label{fig:bsp_julia_notwsplus}
    \end{figure}
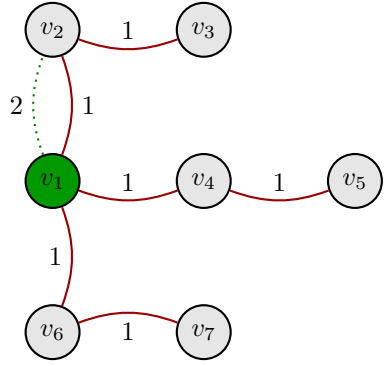
    Consider the graph with ordinal edge costs illustrated in \Cref{fig:bsp_julia} (left):Green dotted edges are in category $\eta_1$ (safe) while red solid edges are in category $\eta_2$ (unsafe); the corresponding edge lengths are specified alongside the edges; and the {demands} of all nodes are $w_i=1$, $v_i\in V, i=1,\dots,7$. Note that this graph is not simple. However, it can be easily transformed into an equivalent simple graph by replacing every green dotted edge by two consecutive green dotted edges, each having half of the length of the original edge, and passing via an additional (artificial) zero-demand node, see \Cref{fig:bsp_julia} (right) for an illustration on a subgraph.
    
    Now let the \emph{new} facility be located at node $v_1$. The solution illustrated in \Cref{fig:bsp_julia_notwsplus} is vertex consistent $C(v_1)$, but not strongly vertex consistent $SC(v_1)$. Consider the excerpt of the solution consisting of the nodes $v_1,v_2$ and $v_3$ with the added auxiliary zero-demand nodes $v_{1'}$ and $v_{2'}$.
    If the solution would be strongly vertex consistent, there would exist a weight vector $\lambda=(\lambda_1,\lambda_2)^\top\in\R^2_\geq$ such that the following inequalities hold (recall that $\At=\bigl(\begin{smallmatrix} 1&1\\0&1\end{smallmatrix}\bigr)$ in the case of an ordinal objective function):
\begin{align}
        2\lambda_1=\lambda^\top\cdot \At\cdot c((1,1',2))&<\lambda^\top\cdot \At\cdot c((1,2))=\lambda_1+\lambda_2,\label{eq:path1}\\
        2\lambda_1+2\lambda_2=\lambda^\top\cdot \At\cdot c((1,2,3))&<\lambda^\top\cdot \At\cdot c((1,1',2,2',3))=4\lambda_1,\label{eq:path2}\\
        2\lambda_1+2\lambda_2=\lambda^\top\cdot \At\cdot c((1,2,3))&<\lambda^\top\cdot \At\cdot c((1,1',2,3))=3\lambda_1+\lambda_2,\\
        2\lambda_1+2\lambda_2=\lambda^\top\cdot \At\cdot c((1,2,3))&<\lambda^\top\cdot \At\cdot c((1,2,2',3))=3\lambda_1+\lambda_2.\label{eq:path24}
    \end{align}
    
    From \Cref{eq:path1,eq:path2} we derive the contradiction $\lambda_1<\lambda_2<\lambda_1$, i.\,e., the solution in \Cref{fig:bsp_julia_notwsplus} is not strongly vertex consistent. However, the solution is vertex consistent $C(v_1)$ with respect to the weight vector $\lambda=(1,1)^\top$. {This is exemplarily shown for $i=2,3$, i.e., for the selected paths to the vertices $v_2,v_3$, by verifying \eqref{eq:path1}-\eqref{eq:path24} with inequalities replacing the strict inequalities.} Similarly, a small computation shows that the solution is consistent \nameref{consistent} with respect to the weight $\lambda=(1,1)^\top$, but not strongly consistent \nameref{stronglyconsistent} {(where the latter follows immediately from the above analysis)}.

    Note that there also exist solutions that are (strongly) vertex consistent, but not (strongly) consistent. In this example, strongly vertex consistent solutions with all green edges and the new facility located in a node other than $v_1$ will always be dominated by the solution with the facility located in $v_1$ and hence, are not (strongly) consistent.  

    Moreover, the paths in (strongly) vertex consistent solutions do not necessarily form a tree structure, as this example (Figure~\ref{fig:bsp_julia_notwsplus}) illustrates.
\end{example}

\subsection{Restriction to Efficient Solutions Attained in Nodes}
For the remainder of this paper, we will restrict ourselves to facility locations at nodes. 
Indeed, it is a well-known result for single-objective median problems that, while points in the interior of edges may be optimal, there always exists an optimal location in a node of the graph \citep[see, e.\,g.,][]{Hakimi64}. In other words, if a point on an edge is optimal, then the nodes incident to this edge are optimal as well. This transfers similarly to ordinal and multi-objective median location problems.

\begin{theorem}[see, e.\,g., \cite{Hakimi64}]\label{thm:nodes}
    For a non-dominated outcome vector of a location problem on graphs with vector-valued edge {costs} and real valued {demands}, e.\,g., \eqref{BLP}, at least one efficient solution \((x,P)\) is located at a node. Moreover, if one point in the interior of an edge is efficient, then the whole edge and both incident vertices are efficient as well.
\end{theorem} 
\begin{proof}
    Suppose that an efficient solution is a point $x$ in the interior of an edge $e= [v_i,v_j]\in E$  with multi-objective/ordinal cost \(c(e)\geq 0\), i.\,e., we can consider $x$ as an artificial node on $e$. We divide the edge $e$ into two edges $e_1=[v_i,x]$ with costs $c(e_1)=\alpha\,c(e)$ and $e_2=[x,v_j]$ with $c(e_2)=(1-\alpha)\,c(e)$ for an appropriate value of $\alpha\in(0,1)$ corresponding to the relative location of \(x\) on the edge \(e\). Let $I\subseteq\{1,\dots,n\}$ define the index set of all nodes for which the selected path from $x$ is routed via node $v_i$. Similarly, we denote by $J\subseteq\{1,\dots,n\}$ the set of indices of all nodes for which the selected path from $x$ is routed via node $v_j$. Obviously, $J= \{1,\dots,n\}\setminus I$ holds. 
    If $\sum_{k\in I} w_k < \sum_{k\in J} w_k$, then $x$ cannot be an efficient solution because placing the facility at node $v_j$ leads to a reduction of the total costs by 
    \[\biggl(\sum_{k\in J} w_k - \sum_{k\in I} w_k\biggr)\,\underbrace{(1-\alpha)}_{>0}\,\underbrace{A\, c(e)}_{\geq 0}\geq0.\]
    If $\sum_{k\in I} w_k > \sum_{k\in J} w_k$ holds, then it follows similarly that $x$ is dominated by $v_i$ as an efficient location. Thus, $\sum_{k\in I} w_k = \sum_{k\in J} w_k$, which implies that the whole edge $e$ including the nodes $v_i$ and $v_j$ is efficient. This concludes the proof.
\end{proof}

\section{Multi-objective Location {and Path Planning} Problems on Specific Graphs}\label{sec:specific}

In this chapter, we investigate special cases of multi-objective location {and path planning} problems on specific types of graphs. We consider line graphs and trees which both have the advantage that all paths between pairs of nodes are unique. Hence, the set of selected paths \(P\) is uniquely determined by the location $\facloc$. {This considerably simplifies the problem since the consistency of the route choices is inherently guaranteed.} This allows us to consider problems with even more than two objective functions {in this section, which we do for the sake of generality.}

{Note that even in the case of unique paths, problem \eqref{BLP}, and its generalization to multi-objective edge costs (MLPP), differs from the ``multi-scenario'' formulation of \cite{Nickel2019Location}, where multi-objective demands rather than multi-objective edge costs are considered. Indeed, while in the former single-objective paths lengths are multiplied by multi-objective demands, in the latter multi-objective path lengths are multiplied by single-objective demands, leading in general to different interpretations, objective vectors, and optimal solutions.}

\subsection{Line Graphs}
Let $G=(V,E)$ be a line graph, i.\,e., the nodes $V=\{v_1,\dots,v_n\}$ are ordered and edges $E=\{e_i=[v_i,v_{i+1}]: i=1,\dots,n-1\}$ only connect consecutive nodes. We consider cost vectors $c(e_i)\in\R^{p}_{\geq}$ for all edges $e_i$, $i=1,\dots,n-1$, a preference matrix $A\in\A_p$, and  {demands} $w_i\geq 0$ for the nodes $v_i$, $i=1,\dots,n$. Note that since we allow zero {demands}, it may happen that several edges (and all incident vertices) are Pareto optimal.
Then, the \emph{multi-objective location {and path planning} problem on line graphs} is defined as
\begin{equation}\label{eq:Line}\tag{L}
	\begin{array}{rl}
		\min & \displaystyle f(z)=\sum_{i=1}^n w_i \, \sum_{k=\min\{i,j\}}^{\max\{i,j\}-1} A\cdot c(e_k)\\[1.5ex]
		\text{s.\,t.} & z=(v_j,P_V(v_j))\in \X \;\text{with}\; j \in \{1,\ldots,n\}.
	\end{array}
\end{equation}
According to Theorem~\ref{thm:nodes}, we can restrict the set of feasible locations to the set of nodes. Obviously, the objective function $f(z)$ can be reformulated for $z=(v_j,P_V(v_j))\in \X$ as 
\[f(z)=\sum_{i=1}^{j-1} A\cdot c(e_i)\, \sum_{k=1}^i w_k + \sum_{i=j}^{n-1} A\cdot c(e_i)\,\sum_{k=i+1}^n w_k .\]
The following theorem {formulates} a necessary and sufficient optimality condition for problem~\eqref{eq:Line}.
\begin{theorem}\label{thm:line}
    The node $v_j$ for $j\in\{1,\dots,n\}$ is the unique Pareto optimal location of \eqref{eq:Line} if and only if $\sum_{i=1}^{j-1}w_i<\frac{1}{2}\sum_{i=1}^{n}w_i$ and $\sum_{i=1}^{j}w_i>\frac{1}{2}\sum_{i=1}^{n}w_i$.
    Moreover, under the assumption that $w_i>0$ for all $i=1,\ldots,n$ the whole edge $e_j$, $j\in\{1,\dots,n-1\}$ (including the incident nodes $v_j$ and $v_{j+1}$) is Pareto optimal if and only if $\sum_{i=1}^{j}w_i=\frac{1}{2}\sum_{i=1}^{n}w_i$.
    
     In the case of nodes with zero {demand}, there may exist many Pareto optimal edges; however, all nodes incident to such edges are also Pareto optimal. With $w_i\geq0$ for all $i=1,\ldots,n$ an edge $e_j$, $j\in\{1,\dots,n-1\}$ (including the incident nodes $v_j$ and $v_{j+1}$) is Pareto optimal if and only if $\sum_{i=1}^{j-1}w_i\leq \frac{1}{2}\sum_{i=1}^{n}w_i$ and $\sum_{i=1}^{j}w_i=\frac{1}{2}\sum_{i=1}^{n}w_i$.
\end{theorem}
\begin{proof}
    The result can be proven analogously to Theorem~\ref{thm:nodes}.
\end{proof}

It is important to note that for line graphs and trees the paths from the {new facility to each} customers are unique. Hence, nothing changes when multiple costs are considered instead of one cost value per edge -- as long as the costs are component-wise non-negative and non-zero. The resulting (Pareto) optimal facility location remains the same as only the {demands} are important. {This highlights again the difference to ``multi-scenario'' location problems with multi-objective demands rather than multi-objective edge costs \citep{Nickel2019Location}, where this is not the case in general.} The following example illustrates this result.

\begin{example}
    Consider the graph in Figure~\ref{fig:line} and assume that all edge costs $c_i\in\R^2_{\geq}$ are non-negative and not equal to the vector of zeros for $i=1,2,3$. 
    Let the costs of the associated multi-objective location problem be given as $A\cdot c_i$ for $i=1,2,3$ {with a preference matrix $A\in\A_2$,} and let the {demands} of the existing facilities be $w_1=1$, $w_2=2$, $w_3=3$ and $w_4=4$. If the new facility is placed at node $v_1$, the costs would be 
    ${(2+3+4)\,A\cdot c_1+(3+4)\,A\cdot c_2+4\,A\cdot c_3}$. However, if the facility is placed at node $v_2$, then only the customers of node $v_1$ have to use edge $[v_1,v_2]$ instead of all customers of nodes $v_2$, $v_3$ and $v_4$. Therefore, the total costs reduce to 
    $A\cdot c_1+(3+4)\,A\cdot c_2+4\,A\cdot c_3$. Similarly, the costs are again lower if the facility is placed at node $v_3$ as $1+2<3+4$. The costs for locating the facility at node $v_4$ would be larger than $1+2+3>4$ and hence, the  (Pareto) optimal location is at node $v_3$ regardless of the multi-objective edge costs 
    $A\cdot c_i\geq 0$, $i=1,2,3$. 

    If the {demands} change, the (Pareto) optimal solution may change. For example, for {demands} $w_1=1$, $w_2=2$, $w_3=3$ and $w_4=6$ the costs of the facility at node $v_3$ would be the same as at node $v_4$. Hence, both nodes and every location on the edge $[v_3,v_4]$ is (Pareto) optimal.
\end{example}
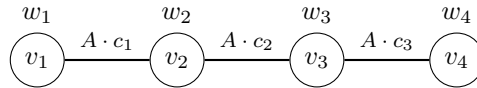
\begin{figure}[ht!]
    \centering
    \clearpage{}\begin{tikzpicture}[xscale=1.85,yscale=1.5,every edge quotes/.append style={font=\footnotesize}]\node[draw,circle,label=above:$w_1$] (1) at (0,0)[] {$v_1$};
\node[draw,circle,label=above:$w_2$] (2) at (1,0) {$v_2$};
\node[draw,circle,label=above:$w_3$] (3) at (2,0) {$v_3$};
\node[draw,circle,label=above:$w_4$] (4) at (3,0) {$v_4$};

\graph {
	(1) --["$A\cdot c_1$",thick,draw=black] (2);
	(2) --["$A\cdot c_2$",thick,draw=black] (3);
	(3) --["$A\cdot c_3$",thick,draw=black] (4);
};
\end{tikzpicture}\clearpage{}
    \caption{Line graph with cost vectors $c_i\in\R_{\geq}^2$, $i=1,2,3$, preference matrix $A\in\A_2$ and {demands} $w_i\in\N_0$, $i=1,2,3,4$.}
    \label{fig:line}
\end{figure}

\subsection{Trees}
In \cite{Shaw1999unified} it is shown that several facility location problems on trees can be reformulated as a tree partitioning problem with a specific structure, and  a generic algorithm to solve location problems on trees is suggested. This work is based on the results in \cite{Cornuejols1983uncapicitated} and \cite{Tamir1996Opn2} and the improved algorithm in \cite{shah02undiscretized}.

The \emph{multi-objective location {and path planning} problem on a {tree}}, i.e., on a cycle-free, connected graph $G=(V,E)$ with {demands} $w_i\geq 0$ for $v_i\in V$, $i=1,\dots,n$, preference matrix $A\in\A_p$, and with edge costs $c(e)\in\R^{p}_\geq$ for $e=[v_i,v_j]\in E$ can be formulated as 
\begin{equation}\label{eq:Tree}\tag{T}
	\begin{array}{rl}
		\min & \displaystyle f(z)=\sum_{i=1}^n w_i\cdot A\cdot c(\selectedPathi)\\
		\text{s.\,t.} & \sol\in \X,
	\end{array}
\end{equation}
where $\selectedPathi=(v=v_{i_1},\dots,v_{i_L}=v_i)$ denotes the \emph{unique} path from the facility location $\facloc$ to node $v_i$ for all $i=1,\dots,n$.
The total cost of this path is given by $c(\selectedPathi)= \sum_{k=1}^{L-1} c([v_{i_k},v_{i_{k+1}}])$.

Problem~\eqref{eq:Tree} can be solved by Algorithm~\ref{alg:tree}, which is an adapted version of the algorithm in \cite{Goldman71} such that all (Pareto) optimal edges are returned, if there is not a single (Pareto) optimal node. The algorithm successively selects leaf nodes and checks whether they accumulate at least half of the total {demand}. Otherwise, it merges the leaf node with its {unique adjacent} node until one node accumulates at least half of the total {demand}. Note that the non-negative edge costs and the preference matrix have no influence on the (Pareto) optimal node.
\begin{algorithm}[ht!]
    \caption{Adaption of the Algorithm of \cite{Goldman71}}\label{alg:tree}
    \SetAlgoLined 
	\KwIn{tree $G=(V,E)$ with $|V|=n$, with {demands} $w_i\geq 0$ for each node $v_i$, $i=1,\dots,n$}
        $W\coloneqq\sum_{i=1}^n w_i$\;
        \While{\textbf{true}}{
        Choose a leaf $v_k\in V$ with incident edge $e=[v_k,v_j]\in E$\;
        \uIf{$w_k>\frac{W}{2}$}{\Return $v_k$ and stop}
        \uElseIf{$w_k=\frac{W}{2}$}{
        $N=\{v_j\}$, $E^*=\{[v_k,v_j]\}$\; \While{$N\neq\emptyset$}{Choose node $v_n\in N$\; \uIf{$w_n=0$}{Add all neighbors of node $v_n$ to $N$ and delete $v_n$ from $N$\; Add all edges, which are incident to $v_n$ to $E^*$} }\Return $E^*$ and stop}
        \uElse{
        $w_j\coloneqq w_j+w_k$\ \tcp*[r]{merge leaf with adjacent node $v_j$}    
        $V\coloneqq V\setminus\{v_k\}$\;
        $E\coloneqq E\setminus\{[v_k,v_j]\}$\;}}
\end{algorithm}
The correctness of the algorithm is based on the fact that the distances $d_T$ satisfy the triangle inequality:
\begin{lem}\label{lem:triangle}
    Let $P_j(v_i)$ denote the uniquely defined path from $v_i\in V$ to $v_j\in V$.
    Then, the distances $d_T(v_i,v_j)=c(P_j(v_i))\in \R^p_\geq$ satisfy the triangle inequality, i.\,e., it holds that 
    $d_T(\facloc,v_i) \leq d_T(\facloc,v_k)+d_T(v_k,v_i)$ for any node $v_k\in V$.
\end{lem}

\begin{figure}[ht!]
 \centering
 \subcaptionbox{Node $v_k$ lies on the path from $\facloc$ to $v_i$. \label{fig:triangle1}}
	[.45\linewidth]{\begin{tikzpicture}[every node/.style={inner sep=0pt, minimum size = 20pt}]
\node[draw,circle] (1) at (0,0)[] {$\facloc$};
\node[draw,circle] (2) at (2,0) {$v_k$};
\node[draw,circle] (3) at (4,0) {$v_i$};

\graph {
(1) --[thick,draw=black] (2);
 (2) --[thick,draw=black] (3);
};
\end{tikzpicture} }\hspace{0.5cm}
        \subcaptionbox{Node $v_k$ lies on the other side of $\facloc$ than $v_i$.\label{fig:triangle2}}
	[.45\linewidth]{\begin{tikzpicture}[every node/.style={inner sep=0pt, minimum size = 20pt}]
\node[draw,circle] (1) at (0,0)[] {$v_k$};
\node[draw,circle] (2) at (2,0) {$\facloc$};
\node[draw,circle] (3) at (4,0) {$v_i$};

\graph {
(1) --[thick,draw=black] (2);
 (2) --[thick,draw=black] (3);
};
\end{tikzpicture} }\\
	\subcaptionbox{Node $v_k$ lies on the other side of $v_i$ than $\facloc$.\label{fig:triangle3}}
	[.45\linewidth]{\begin{tikzpicture}[every node/.style={inner sep=0pt, minimum size = 20pt}]
\node[draw,circle] (1) at (0,0)[] {$\facloc$};
\node[draw,circle] (2) at (2,0) {$v_i$};
\node[draw,circle] (3) at (4,0) {$v_k$};

\graph {
(1) --[thick,draw=black] (2);
 (2) --[thick,draw=black] (3);
};
\end{tikzpicture} }\hspace{0.5cm}
	\subcaptionbox{The shortest path from $\facloc$ to $v_k$ uses at least one but not all edges of the path $\selectedPathi$\label{fig:triangle4}}
	[.45\linewidth]{\begin{tikzpicture}[every node/.style={inner sep=0pt, minimum size = 20pt}]
\node[draw,circle] (1) at (0,0)[] {$\facloc$};
\node[draw,circle] (2) at (2,0) {$v_h$};
\node[draw,circle] (4) at (2,-1) {$v_k$};
\node[draw,circle] (3) at (4,0) {$v_i$};

\graph {
(1) --[thick,draw=black] (2);
 (2) --[thick,draw=black] (3);
  (2) --[thick,draw=black] (4);
};
\end{tikzpicture} }
	\caption{All possibilities how a node $v_k\in V\setminus\{\facloc,v_i\}$ can lie in relation to the path $\selectedPathi$ from $v_i$ to $\facloc$ in a tree.}
	\label{fig:trianlge}
\end{figure}
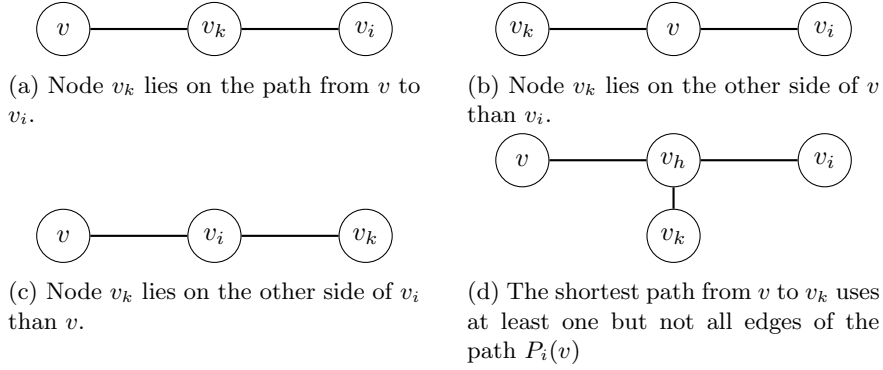

\begin{proof}
There exist four different possibilities on how the node $v_k$ lies in relation to the path from $\facloc$ to $v_i$, which are visualized in Figure~\ref{fig:trianlge}. The node $v_k$ can either be part of the path from $\facloc$ to $v_i$, see Figure~\ref{fig:triangle1}, then it holds that $d_T(\facloc,v_i)= d_T(\facloc,v_k)+d_T(v_k,v)$. 

For the following cases we always use the fact that all components of the edge costs are greater or equal to zero, i.\,e., $c(e)\in\R_\geq^p$ for all $e\in\E$.
 If $\facloc$ is on the path from $v_k$ to $v_i$, see Figure~\ref{fig:triangle2}, then it follows that $d_T(\facloc,v_i)\leq d_T(v_k,v_i)\leq d_T(\facloc,v_k)+d_T(v_k,v_i)$.
 
 If $v_i$ is included in the path from $\facloc$ to $v_k$, see Figure~\ref{fig:triangle3}, then $d_T(\facloc,v_i)\leq d_T(\facloc,v_k)\leq d_T(\facloc,v_k)+d_T(v_k,v_i)$. 

In the last case, see Figure~\ref{fig:triangle4}, there exists a node $v_h$ with $v_h\neq \facloc$, $v_h\neq v_k$ and $v_h\neq v_i$ in the unique path from $\facloc$ to $v_k$ such that all nodes on the path from $v_h$ to $v_i$ are not included in the path from $\facloc$ to $v_k$. Then it holds that $d_T(\facloc,v_i)= d_T(\facloc,v_h) + d_T(v_h,v_i) \leq d_T(\facloc,v_h) +2\, d_T(v_h,v_k) + d_T(v_h,v_i)=d_T(\facloc,v_k)+d_T(v_k,v_i)$, which concludes the proof.
\end{proof}

\begin{theorem}
    Algorithm~\ref{alg:tree} computes the set of (Pareto) optimal locations of the {multi-objective} location {and path planning} problem~\eqref{top1} on a tree.
\end{theorem}
\begin{proof}
    This follows in analogy to the single-objective median location problems on trees, see \cite{Goldman71}, as the {multi-objective} distances on trees still satisfy the triangle inequality, see~Lemma~\ref{lem:triangle}.
\end{proof}

\section{Solution Strategies for Bi-objective Location {and Path Planning} Problems on {General} Graphs}\label{sec:generalgraphs}
For general graphs, we restrict ourselves to bi-objective location {and path planning} problems {in order to keep the concept of consistent route choices and associated preferences simple. Indeed},  a consistency definition with $p>2$ objectives would require the definition of marginal trade-offs between all pairs of objectives. {Such pairwise trade-offs} would restrict the {original cone} $\R^p_\geq$ of possible weight {vectors $\lambda$ in Definition~\ref{def:consistency}} in a way that {it may have} more {than $p$} extreme rays, {i.e., more} than the dimension of the outcome space. This would, in general, change the dimension of the preference matrix $A$, that {represents} the cone {of possible weight vectors $\lambda$} by its facets. 
{For a discussion of} the case of considering only weights between consecutive objective functions, the transformation matrix $A$ is introduced in \cite{klamroth2026adapting}.

Let $G=(V,E)$ be a {simple and connected} graph with bi-objective edge costs $c(e)\in\R^2_\geq$. 
Recall that the bi-objective shortest paths between the new facility and the customers are, in general, not unique. To solve the bi-objective location and path planning problem \eqref{BLP} without any consistency requirements w.\,r.\,t.\ the path choices, the two-phase algorithm introduced in \cite{Skriver2004Bicriteria} can be applied. However, since we focus on solutions for which the path choices are consistent over all customers, this method is not applicable here.
In order to incorporate the consistency requirements, we propose an alternative algorithm that iteratively selects paths and adapts the preference matrix $A$, {iteratively} including the preference information {that is induced by the partial route choice decisions taken so far} for the subsequent path selections. 
To incorporate such preference information  in the preference matrix, we set
\begin{equation}\label{Awg}
    A={\begin{pmatrix}
        a_{11}&a_{12}\\ a_{21}&a_{22}
    \end{pmatrix}=}\begin{pmatrix}
        1&\beta\\ \g&1
    \end{pmatrix} \in\A_2,
\end{equation}
initialized with $A^0=(\begin{smallmatrix} 1 & \beta_0 \\\gamma_0 &1 \end{smallmatrix})=(\begin{smallmatrix} 1 & 0 \\0 &1 \end{smallmatrix})=I$ for general bi-objective problems, and initialized with $A^0=(\begin{smallmatrix} 1 & \beta_0 \\\gamma_0 &1 \end{smallmatrix})=(\begin{smallmatrix} 1 & 1 \\0 & 1 \end{smallmatrix})=\At$ for ordinal optimization problems. In order to obtain non-contradicting preferences, we ensure that $\beta\cdot\g\leq1$ at all times. If there exists only one efficient $(\facloc,v_i)$-path, we do not gain any preference information by selecting this specific path; hence, we do not update the parameters in the preference matrix. Otherwise, if there exist multiple efficient $(\facloc,v_i)$-paths, selecting one path yields a preference information that can be incorporated into the preference matrix $A$. {Overall, we aim at the computation of \emph{all} (extreme) supported non-dominated outcome vectors in the objective space together with a minimum complete set of Pareto optimal solutions, i.e., one location and routing solution per outcome vector. Hence, we apply backtracking to generate a minimum complete set of Pareto optimal location and routing solutions.} 

{In the following, we discuss the update scheme for the preference matrix $A$ given some (additional) preference information. Towards this end,} 
let $\Path$ and $\AltPath$ be two efficient paths w.\,r.\,t.\ the objective function $A^0\cdot c(\cdot)$, which implies that $A^0 c(\Path)\nleq A^0 c(\AltPath)$ and $A^0 c(\AltPath)\nleq A^0 c(\Path)$ and {$A^0 c(\Path)\neq A^0 c(\AltPath)$}. If we select path $\Path$, we have to adapt the values $\beta$ and $\gamma$ such that $\Path$ dominates {(or is equivalent to)} $\AltPath$ in the {updated} objective function $\Aup{\cdot c(\cdot)}$, i.\,e., $\Aup c(P)\leq \Aup c(Q)$. Since $A^0 c(\Path)$ and $A^0 c(\AltPath)$ are pairwise non-dominated, we distinguish two cases:

{Case 1:} 
\begin{align}
    &c_1(\Path) + \beta_0\, c_2(\Path) > c_1(\AltPath) + \beta_0\, c_2(\AltPath), \quad\text{{and}} \label{eq:pref_beta0}\\
    &\g_0\, c_1(\Path) + c_2(\Path) < \g_0\, c_1(\AltPath) + c_2(\AltPath). \label{eq:pref_g0}
\end{align}
Consequently, we only need to update $\beta$, since \eqref{eq:pref_g0} already complies with the preference of \(\Path\) {over $\AltPath$ with the current value of} \(\g_0\). Under the assumptions that the paths $P$ and $Q$ are pairwise non-dominated w.\,r.\,t.\ $A^{{0}}\, c(\cdot)$ and that $\beta_0\cdot\g_0\leq 1$, equations~\eqref{eq:pref_beta0} and \eqref{eq:pref_g0} are equivalent to $c_1(\Path)>c_1(\AltPath)$ and $c_2(\Path)<c_2(\AltPath)$. 
The new value of $\beta$ should satisfy:
\begin{align}
    & c_1(\Path) + \beta\, c_2(\Path) \leq c_1(\AltPath) + \beta\, c_2(\AltPath) \notag\\
    \iff & \frac{c_1(\Path)-c_1(\AltPath)}{c_2(\AltPath)-c_2(\Path)}\leq \beta. \label{preference1}
\end{align}
Inequality \eqref{preference1} defines the marginal ratio such that $\Path$ is preferred over $\AltPath$. {Hence,} if $c_1(\Path) + \beta_0\, c_2(\Path) \geq c_1(\AltPath) + \beta_0\, c_2(\AltPath)$ we update $\beta$ to
\begin{align}\label{eq:beta}
    \tilde{\beta}= \frac{c_1(\Path)-c_1(\AltPath)}{c_2(\AltPath)-c_2(\Path)} 
    \intertext{and set}
    \Aup=\begin{pmatrix}    1 & \tilde{\beta}\\ \g & 1\end{pmatrix}.\notag
\end{align}

{Case 2:}
\begin{align}
    &c_1(\Path) + \beta_0\, c_2(\Path) < c_1(\AltPath) + \beta_0\, c_2(\AltPath), \quad\text{{and}} \label{eq:pref2_beta0}\\
    &\g_0\, c_1(\Path) + c_2(\Path) > \g_0\, c_1(\AltPath) + c_2(\AltPath). \label{eq:pref2_g0}
\end{align}
In this case, we only update $\g$, as \eqref{eq:pref2_beta0} {complies with} the preference of $\Path$ over $\AltPath$ for {the current value of} $\beta_0$. {Using the same arguments as above,} inequalities \eqref{eq:pref2_beta0} and \eqref{eq:pref2_g0} are equivalent to $c_1(\Path)<c_1(\AltPath)$ and $c_2(\Path)>c_2(\AltPath)$. Thus, $\g$ should satisfy
\begin{align}
     & \g\, c_1(\Path) + c_2(\Path) \leq \g\, c_1(\AltPath) + c_2(\AltPath) \notag\\
    \iff &\frac{c_2(\Path)-c_2(\AltPath)}{c_1(\AltPath)-c_1(\Path)} \leq \g. \label{preference2}
\end{align}
{Hence,} if $\g_0\, c_1(\Path) + c_2(\Path) \geq \g_0\, c_1(\AltPath) + c_2(\AltPath)$ we update $\g$ to
\begin{align}\label{eq:gamma}
    \tilde{\g}= \frac{c_2(\Path)-c_2(\AltPath)}{c_1(\AltPath)-c_1(\Path)}
    \intertext{and obtain the updated preference matrix}
    \Aup=\begin{pmatrix} 1 & \beta\\ \tilde{\g} & 1 \end{pmatrix}.\notag
\end{align}
{Then, after selecting path $\Path$ over path $\AltPath$,} we replace the objective function with $\Aup\cdot c(\cdot)$ when evaluating facility-to-customer paths with respect to the other nodes $\facloc'\in V\setminus\{\facloc,v_i\}$. By iteratively updating the preference matrix, we are thereby able to construct vertex consistent solutions. {Backtracking, i.e., eventually considering the alternative preference of selecting path $\AltPath$ over path $\Path$, yields a complete set of (extreme) supported non-dominated outcome vectors. Note that, as a by-product,} we implicitly determine the intervals of the weight space decomposition.

\subsection{Algorithmic Approach}\label{sec:alg}

We introduce an algorithm that determines {a minimum complete set of,} both, strongly consistent~\nameref{stronglyconsistent} and consistent~\nameref{consistent} solutions for the bi-objective location and path planning problem (\ref{BLP}). While the two-phase algorithm introduced in \cite{Skriver2004Bicriteria} can be adapted to determine the extreme supported, i.\,e., strongly consistent {outcome vectors}, {it can,} in general, not determine all {outcome vectors} that are consistent but not strongly consistent. {To overcome these shortcomings}, we first propose a backtracking algorithm to determine {a minimum complete set of} strongly consistent solutions, where the implied preference {is} derived from the preference matrix. Then, we extend {this} algorithm in order to derive {a minimum complete set of} consistent solutions in \Cref{app:nonex}. Note that the fact that our algorithm can easily be adapted to obtain {a minimum complete set of} consistent solutions is an advantage over dichotomic search based approaches.

{The general framework is outlined in \Cref{alg:DP}, and the implementation of associated path computations under consistency requirements is detailed in \Cref{alg:CSP}.}

\begin{algorithm}[h!]
\caption{Node-based Algorithm}\label{alg:DP}
    \KwIn{graph $G=(V,E)$ with $|V|=n$, objective function $c:\X\rightarrow \R^2_{\geq}$, initial preference matrix $A_{}^{{0}}\in\A_2$}
    \KwOut{{minimum complete} set of strongly consistent solutions $\SCon$ }
    $S_v\coloneqq \emptyset\quad \forall\:v\in V$\;
    \Forall{$\facloc\in V$}{
    Compute the sets of shortest paths $\ivPaths$ from $\facloc$ to $v_i$ for all $v_i\in V\setminus\{\facloc\}$ with objective function $A_{}^{{0}}\cdot c(\cdot)$\;
    $\allPaths\coloneqq(\mathcal{P}_1(v),\dots,\mathcal{P}_n(v))$\;
$\VSCon\coloneqq \text{sCSP}(G,c,\facloc,\allPaths,A_{}^{{0}},(\facloc,\emptyset),\emptyset)$\;
\Forall{$\sol\in\VSCon$}{
    $f(z)\coloneqq A_{}^{{0}} \cdot c(\selectedPaths)$\;
    $S_v\coloneqq S_v\cup(f(z),z)$\;
}
    }
    Filter the set $\SCon\coloneqq \bigcup_{\facloc\in V}S_v$ wrt.\ Pareto dominance\\
    \Return $\SCon$
\end{algorithm} 
\Cref{alg:DP} is a node-based algorithm that returns {a minimum complete} set of strongly consistent solutions \nameref{stronglyconsistent}. It considers each node $\facloc\in V$ as a possible facility location and computes {a minimum complete set of Pareto optimal} facility-to-customer paths from $\facloc$ to every other node $v_i\in V$ in $G$ with respect to the objective function $A_0\cdot c(\cdot)$. 
{A minimum complete set of Pareto optimal} paths can be computed using the multi-objective Dijkstra algorithm proposed in \cite{MARISTANYDELASCASAS2021105424}.
The resulting sets of {initial Pareto optimal} paths $\ivPaths$, $i=1,\dots,n$, are stored in $\allPaths$. {Note that these paths are used for initialization purposes and to simplify further computations. However, not every combination of such paths reflects consistent preferences in general.}

Then, \Cref{alg:DP} determines the set of strongly vertex consistent solutions $\VSCon$ {by recursively calling} \Cref{alg:CSP}, computes the objective function values $f(z)$ for all $\sol\in\VSCon$ and stores them in a set $S_v$.
Finally, the union of all sets $S_v$ is filtered w.\,r.\,t.\ Pareto dominance to obtain {a minimum complete} set of strongly consistent solutions $\SCon$. An efficient algorithm for the Pareto filtering of the union of non-dominated sets can be found in \cite{KLAMROTH2024106506}.

\begin{algorithm}[h!]
\caption{Strongly Consistent Shortest Paths \\ \hspace*{6.8em}${\text{sCSP}(G,c,\facloc,\allPaths,A,G_c,\VSCon)}$}\label{alg:CSP}
    \KwIn{graph $G=(V,E)$, objective function $c:\X\rightarrow \R^2_{\geq}$, facility location $\facloc\in V$, $n$-tuple $\allPaths$ containing {Pareto optimal} paths, preference matrix $A\in\A_2$, {graph representing already selected paths} $G_c=\parS$, 
    {partial} set of strongly vertex consistent solutions $\VSCon$ for the facility location $\facloc$}
    \KwOut{{minimum complete} set of strongly vertex consistent solutions $\VSCon$ for the facility location $\facloc$}
Choose $v_i\in V\setminus\Vc$ using a suitable selection strategy\;
Extract the $v$-$v_i$-paths $\ivPaths$ from $\allPaths$ \; 
        Filter $\ivPaths$ for non-dominance wrt.\ $A\cdot c{(\cdot)}$ \;
        \Forall{$\Path\in\ivPaths$ with $\Path=(\VP,\EP)$}{
            $\Aup\coloneqq \text{update}(A,\Path,\ivPaths,c)$\;
            \If{$\aup_{1,2}\cdot \aup_{2,1}<1$}{
                \If{$\Vc\cup\VP\neq V$}{
                $\text{sCSP}(G,c,\facloc,\allPaths,\Aup,(\Vc\cup\VP,\Ec\cup\EP),\VSCon)$\;
                }
                \Else{
                    {Extract the set $\selectedPaths$  from $(\Vc\cup\VP,\Ec\cup\EP)$}\;
                    $\VSCon\coloneqq\VSCon\cup{\{(\facloc,\selectedPaths)\}}$\;
                }
            }    
        }
\Return $\VSCon$\;
\end{algorithm}

\Cref{alg:CSP} implements the iterative construction of strongly vertex consistent solutions $\VSCon$ for a given facility location $v\in V$. {For this purpose, we iteratively update} {a subgraph $G_c=\parS$ of $G$ that reflects the partial path choices, together with} the {current status of the} preference matrix $A\in\A_2$.
{To extend such a partial solution with paths to customers that are not yet covered,} we choose a node $v_i\in V\setminus\{V_c\}$ using a suitable selection strategy and extract the set of {Pareto optimal} paths $\ivPaths$ to node $v_i$, {which we use as candidate paths to reach node $v_i$}. The significance of the selection process is discussed in \Cref{rem:selectionstrategy}.

\begin{algorithm}[ht!]

\caption{$\update(A,\Path,\ivPaths,c)$}\label{alg:update}
    \KwIn{preference matrix $A\in\A_2$, selected path $\Path\in\ivPaths$, set of paths $\ivPaths$, objective function $c:\X\rightarrow \R^2_{\geq}$}
\KwOut{updated preference matrix $A\in\A_2$}
    $\beta\coloneqq {a}_{1,2}$\;
    $\g\coloneqq {a}_{2,1}$\;
    \Forall{$\AltPath\in \ivPaths\setminus\{\Path\}$}{
            \uIf{$c_1(\Path)<c_1(\AltPath)$}{
$\gamma\coloneqq \max\Big\{\gamma,\frac{c_2(\Path)-c_2(\AltPath)}{c_1(\AltPath)-c_1(\Path)}\Big\}$\tcp*[r]{\Cref{eq:gamma}}\;
                \vspace{-1em}
                }
                \ElseIf{$c_2(\Path)<c_2(\AltPath)$}{
                $\beta\coloneqq \max\Big\{\beta,\frac{c_1(\Path)-c_1(\AltPath)}{c_2(\AltPath)-c_2(\Path)}\Big\}$\tcp*[r]{\Cref{eq:beta}}\;
                \vspace{-1em}
}
        }
    ${a}_{1,2}\coloneqq \beta$\;
    ${a}_{2,1}\coloneqq \g$\;
    \Return $A$\;
\end{algorithm} 
For every path $\Path\in\ivPaths$ {that is selected in Line~4 of \Cref{alg:CSP}}, we update the preference matrix~$A$ using \Cref{alg:update}. Here, we use \Cref{eq:beta,eq:gamma} to update the parameters of the preference matrix. Note that as we might compare several {Pareto optimal} paths, we take the maximum over all possible updates for $\beta$ and $\gamma$. We then check whether the updated preference matrix $\Aup$ is feasible by evaluating the condition $\beta\cdot\gamma<1$. By demanding a strict inequality, we exclude solutions that are vertex consistent but not strongly vertex consistent. 
If the updated preference matrix $\Aup$ is infeasible (i.\,e., represents contradicting preferences), we discard the path $\Path$ {(since it can not be used to consistently extend the current partial solution)} and continue with another path $\Path\in\ivPaths$. Otherwise, if there are nodes that have not been {covered by} the {current partial} solution, we iteratively extend the solution by nodes and edges using \Cref{alg:CSP} with the updated matrix $\Aup$ and the partial solution defined by path $\Path=(\VP,\EP)$. In every consecutive iteration, we select another, not yet covred node $v_i\in V$ with the chosen selection strategy and evaluate the {Pareto optimal} paths with respect to the updated objective function $\Aup\cdot c (\cdot)$, incorporating the preference information derived from earlier path choices. We repeat these steps until the merged graph $(\Vc\cup\VP,\Ec\cup\EP)$ consisting of the partial strongly vertex consistent {paths stored in} $(\Vc,\Ec)$ and the added path $P=(\VP,\EP)$ covers all nodes. Using the updated preference matrix $\Aup$ we guarantee that every path selection is consistent with all previously selected paths. As soon as all nodes are covered by the solution {the corresponding paths to all customers can be extracted from the merged graph $(\Vc\cup\VP,\Ec\cup\EP)$}, \Cref{alg:CSP} returns the set of all strongly vertex consistent solutions $\VSCon$.

\begin{example}\label{ex:missing_nonextreme}
    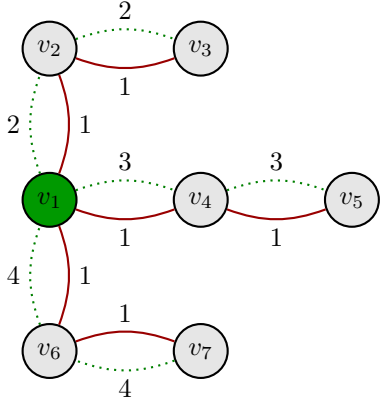
\begin{figure}[!ht]
        \centering
        \begin{tikzpicture}

\node[draw,circle,thick,fill=black!40!green] (1) at (0,0)[] {$v_1$};
\node[draw,circle,thick,fill=black!10] (2) at (0,2) {$v_2$};
\node[draw,circle,thick,fill=black!10] (3) at (2,2) {$v_3$};
\node[draw,circle,thick,fill=black!10] (4) at (2,0) {$v_4$};
\node[draw,circle,thick,fill=black!10] (5) at (4,0) {$v_5$};
\node[draw,circle,thick,fill=black!10] (6) at (0,-2) {$v_6$};
\node[draw,circle,thick,fill=black!10] (7) at (2,-2) {$v_7$};

\graph {
(1) --[bend right=20,thick,draw=black!40!red,swap,"1"] (2);
(1) --[bend left=20,dotted,draw=black!50!green,thick,"2"] (2);
(2) --[bend right=20,thick,draw=black!40!red,swap,"1"] (3);
(2) --[bend left=20,dotted,draw=black!50!green,thick,"2"] (3);
(6) --[bend right=20,thick,draw=black!40!red,swap,"1"] (1);
(6) --[bend left=20,dotted,draw=black!50!green,thick,"4"] (1);
(7) --[bend right=20,thick,draw=black!40!red,swap,"1"] (6);
(7) --[bend left=20,dotted,draw=black!50!green,thick,"4"] (6);
(1) --[bend right=20,thick,draw=black!40!red,swap,"1"] (4);
(1) --[bend left=20,dotted,draw=black!50!green,thick,"3"] (4);
(4) --[bend right=20,thick,draw=black!40!red,swap,"1"] (5);
(4) --[bend left=20,dotted,draw=black!50!green,thick,"3"] (5);

};

\end{tikzpicture}

         \caption{Ordinal graph $G=(V,E)$ with the facility location at node $v_1$, {c.f.\ Figure~\label{fig:bsp_julia}. All demands are supposed to be equal to one, i.e., $w_i=1$ for all $i=1,\dots,7$.}}
        \label{fig:bsp_julia2}
    \end{figure}

    We illustrate the idea of \Cref{alg:CSP,alg:update,alg:DP} using the ordinal graph depicted in \Cref{fig:bsp_julia2} as introduced in \Cref{ex:julia}.
    Let $\facloc=v_1$ be the {selected} facility location and initialize the preference matrix as $\At=\bigl(\begin{smallmatrix} 1 & 1\\ 0& 1 \end{smallmatrix}\bigr)$. 
    Choose $v_i=v_2$. There are two path alternatives with counting vectors $(0,1)^\top$ and $(2,0)^\top$, 
    and both of them are eligible path choices since
    \begin{align*}
        \begin{pmatrix}
            1 & 1\\0&1
        \end{pmatrix}\cdot\begin{pmatrix}
            0\\1
        \end{pmatrix}&=\begin{pmatrix}
            1\\1
        \end{pmatrix}
        \quad\text{and}\quad 
        \begin{pmatrix}
            1 & 1\\0&1
        \end{pmatrix}\cdot\begin{pmatrix}
            2\\0
        \end{pmatrix}=\begin{pmatrix}
            2\\0
        \end{pmatrix}.
    \end{align*}
    \begin{figure}[b!]
        \centering
        \begin{tabular}{cc@{\hspace{0.5cm}}cc}
        $\begin{pmatrix}
            9\\9
        \end{pmatrix}$&\resizebox{0.35\linewidth}{!}{
        \begin{tikzpicture}
\useasboundingbox (-0.25,-2.25) rectangle (4.5,2.5);
\node[draw,circle,thick,fill=black!40!green] (1) at (0,0)[] {$v_1$};
\node[draw,circle,thick,fill=black!10] (2) at (0,2) {$v_2$};
\node[draw,circle,thick,fill=black!10] (3) at (2,2) {$v_3$};
\node[draw,circle,thick,fill=black!10] (4) at (2,0) {$v_4$};
\node[draw,circle,thick,fill=black!10] (5) at (4,0) {$v_5$};
\node[draw,circle,thick,fill=black!10] (6) at (0,-2) {$v_6$};
\node[draw,circle,thick,fill=black!10] (7) at (2,-2) {$v_7$};

\graph {
(1) --[bend left=20,thick,draw=black!40!red,"1"] (2);
(2) --[bend left=20,thick,draw=black!40!red,"1"] (3);
(6) --[bend left=20,thick,draw=black!40!red,"1"] (1);
(7) --[bend left=20,thick,draw=black!40!red,"1"] (6);
(1) --[bend left=20,thick,draw=black!40!red,"1"] (4);
(4) --[bend left=20,thick,draw=black!40!red,"1"] (5);

};

\end{tikzpicture}

 }&
        $\begin{pmatrix}
            12\\6
        \end{pmatrix}$&\resizebox{0.35\linewidth}{!}{
        \begin{tikzpicture}
\useasboundingbox (-0.25,-2.25) rectangle (4.5,2.5);
\node[draw,circle,thick,fill=black!40!green] (1) at (0,0)[] {$v_1$};
\node[draw,circle,thick,fill=black!10] (2) at (0,2) {$v_2$};
\node[draw,circle,thick,fill=black!10] (3) at (2,2) {$v_3$};
\node[draw,circle,thick,fill=black!10] (4) at (2,0) {$v_4$};
\node[draw,circle,thick,fill=black!10] (5) at (4,0) {$v_5$};
\node[draw,circle,thick,fill=black!10] (6) at (0,-2) {$v_6$};
\node[draw,circle,thick,fill=black!10] (7) at (2,-2) {$v_7$};

\graph {
(1) --[bend right=20,dotted,draw=black!50!green,thick,swap,"2"] (2);
(2) --[bend right=20,dotted,draw=black!50!green,thick,swap,"2"] (3);
(6) --[bend left=20,thick,draw=black!40!red,"1"] (1);
(7) --[bend left=20,thick,draw=black!40!red,"1"] (6);
(1) --[bend left=20,thick,draw=black!40!red,"1"] (4);
(4) --[bend left=20,thick,draw=black!40!red,"1"] (5);

};

\end{tikzpicture}
 }\\[0.75cm]
        $\begin{pmatrix}
            18\\3
        \end{pmatrix}$&\resizebox{0.35\linewidth}{!}{
        \begin{tikzpicture}
\useasboundingbox (-0.25,-2.25) rectangle (4.5,2.5);
\node[draw,circle,thick,fill=black!40!green] (1) at (0,0)[] {$v_1$};
\node[draw,circle,thick,fill=black!10] (2) at (0,2) {$v_2$};
\node[draw,circle,thick,fill=black!10] (3) at (2,2) {$v_3$};
\node[draw,circle,thick,fill=black!10] (4) at (2,0) {$v_4$};
\node[draw,circle,thick,fill=black!10] (5) at (4,0) {$v_5$};
\node[draw,circle,thick,fill=black!10] (6) at (0,-2) {$v_6$};
\node[draw,circle,thick,fill=black!10] (7) at (2,-2) {$v_7$};

\graph {
(1) --[bend right=20,dotted,draw=black!50!green,thick,swap,"2"] (2);
(2) --[bend right=20,dotted,draw=black!50!green,thick,swap,"2"] (3);
(1) --[bend right=20,dotted,draw=black!50!green,thick,swap,"3"] (4);
(4) --[bend right=20,dotted,draw=black!50!green,thick,swap,"3"] (5);
(1) --[bend right=20,thick,draw=black!40!red,"1"] (6);
(6) --[bend right=20,thick,draw=black!40!red,"1"] (7);

};

\end{tikzpicture}
 }&
        $\begin{pmatrix}
            27\\0
        \end{pmatrix}$&\resizebox{0.35\linewidth}{!}{
        \begin{tikzpicture}
\useasboundingbox (-0.25,-2.25) rectangle (4.5,2.5);
\node[draw,circle,thick,fill=black!40!green] (1) at (0,0)[] {$v_1$};
\node[draw,circle,thick,fill=black!10] (2) at (0,2) {$v_2$};
\node[draw,circle,thick,fill=black!10] (3) at (2,2) {$v_3$};
\node[draw,circle,thick,fill=black!10] (4) at (2,0) {$v_4$};
\node[draw,circle,thick,fill=black!10] (5) at (4,0) {$v_5$};
\node[draw,circle,thick,fill=black!10] (6) at (0,-2) {$v_6$};
\node[draw,circle,thick,fill=black!10] (7) at (2,-2) {$v_7$};

\graph {
(1) --[bend right=20,dotted,draw=black!50!green,thick,swap,"2"] (2);
(2) --[bend right=20,dotted,draw=black!50!green,thick,swap,"2"] (3);
(6) --[bend right=20,dotted,draw=black!50!green,thick,swap,"4"] (1);
(7) --[bend right=20,dotted,draw=black!50!green,thick,swap,"4"] (6);
(1) --[bend right=20,dotted,draw=black!50!green,thick,swap,"3"] (4);
(4) --[bend right=20,dotted,draw=black!50!green,thick,swap,"3"] (5);

};

\end{tikzpicture}
 }
        \end{tabular}
        \caption{Extreme supported, strongly consistent solutions and their outcome vectors obtained with \Cref{alg:DP,alg:CSP,alg:update} for \Cref{ex:missing_nonextreme} {for $v_1$ as the selected facility}.}
        \label{fig:bsp_julia_nonextreme_all}
    \end{figure}
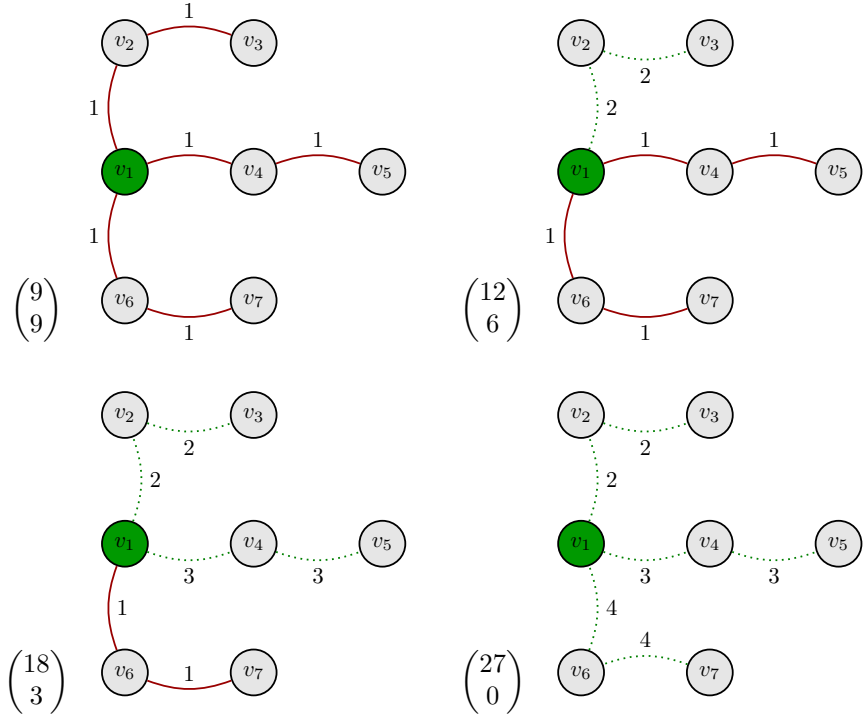    
    If we select the red edge, we update the preference matrix according to \eqref{preference2}:
    \begin{equation*}
        \Aup=\begin{pmatrix}
            1&1\\\frac{1}{2}&1
        \end{pmatrix}.
    \end{equation*}
    In the following, every {consistent} path choice is uniquely {defined}: Since $\Aup$ expresses that one red edge is preferred over two green edges, one red edge is also preferred over three, four and more green edges. Hence, the matrix is not updated in any of the subsequent iterations and we derive the strongly consistent solution with outcome vector $(9,9)^\top$, see \Cref{fig:bsp_julia_nonextreme_all}.

        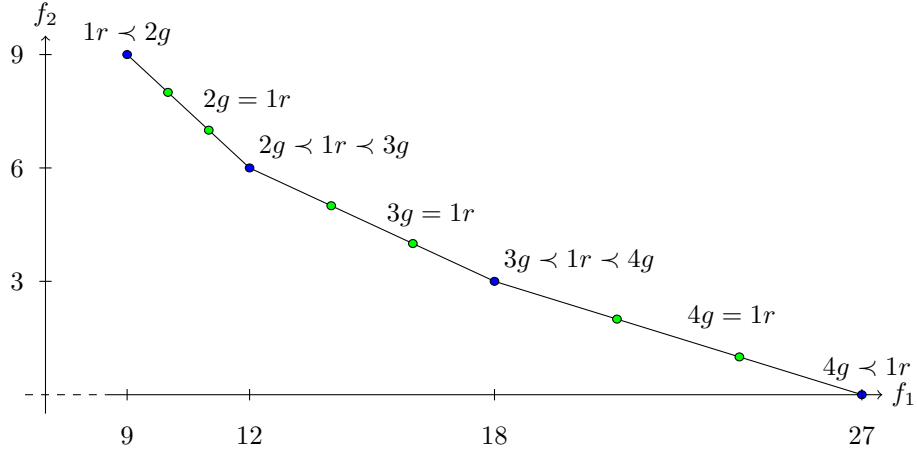
\begin{figure}[!ht]
        \centering
        \begin{tikzpicture}[xscale=0.54,yscale=0.5]

\draw[-,dashed] (6.5,0) -- (8.5,0);
\draw[->] (8.5,0) -- (27.5,0) node[right] {$f_1$};
\draw[->] (7,-0.5) -- (7,9.5) node[above] {$f_2$};

\draw plot coordinates {(9,9) (12,6) (18,3) (27,0)};

\draw[fill=blue](9,9)circle(3pt) node[above] {$1r\prec2g$};
\draw[fill=blue](12,6)circle(3pt) node[above right] {$2g\prec1r\prec3g$};
\draw[fill=blue](18,3)circle(3pt) node[above right] {$3g\prec1r\prec4g$};
\draw[fill=blue](27,0)circle(3pt) node[above,yshift=2pt,xshift=2pt] {$4g\prec1r$};

\draw[fill=green](10,8)circle(3pt);
\draw[fill=green](11,7)circle(3pt);
\draw[fill=green](14,5)circle(3pt);
\draw[fill=green](16,4)circle(3pt);
\draw[fill=green](21,2)circle(3pt);
\draw[fill=green](24,1)circle(3pt);

\path (10,8) -- (11,7) coordinate[midway] (M1);
\draw (M1) node[above right,yshift=-4pt,xshift=2pt] {$2g=1r$};
\path (14,5) -- (16,4) coordinate[midway] (M2);
\draw (M2) node[above right,yshift=-4pt,xshift=2pt] {$3g=1r$};
\path (21,2) -- (24,1) coordinate[midway] (M3);
\draw (M3) node[above right] {$4g=1r$};

\foreach \x in {9,12,18,27}
{
    \draw (\x,0.15) -- (\x,-0.15); \draw (\x, -0.6) node[below] {\x}; }

\foreach \y in {3,6,9}
{
    \draw (6.85,\y) -- (7.15,\y); \draw (6.7, \y) node[left] {\y}; }
    
\end{tikzpicture}         \caption{Non-dominated outcome vectors and their associated preferences for \Cref{ex:missing_nonextreme}. Extreme supported solutions are illustrated in blue, non-extreme supported solutions in green.}
        \label{fig:bsp_julia_outcome}
    \end{figure}
    
    If we select the green edge as path to node $v_2$, we update the preference matrix according to \eqref{preference1}:
    \begin{equation}\label{eq:matrix_Aup}
        \Aup=\begin{pmatrix}
            1&2\\0&1
        \end{pmatrix}
    \end{equation}
    and continue to evaluate the path alternatives to node $v_4$. Regarding the updated preference matrix, the path using the red edge as well as the path using the green edge are feasible.
    Selecting the red edge, we consider the updated preference matrix $$\Aup=\begin{pmatrix}
        1&2\\ \frac{1}{3}&1
    \end{pmatrix}$$ and all subsequent path choices are unique again. To nodes $v_5,v_6$ and $v_7$ we use the respective red edges, while we use the green edges to node $v_3$ and derive a solution with outcome vector $(12,6)^\top$.
    
    Selecting the alternative path using the green edge to node $v_4$, both path alternatives to node $v_6$ are feasible again. Both choices lead to an update of the preference matrix, which is given by $\Aup=\bigl(\begin{smallmatrix}  1&3\\ \frac{1}{4}&1 \end{smallmatrix}\bigr)$ and $\Aup=\bigl(\begin{smallmatrix}  1&4\\0&1 \end{smallmatrix}\bigr)$, respectively. All remaining path choices are unique. Hence, we derive two further solutions with outcome vectors $(18,3)^\top$ and $(27,0)^\top$.
    All four strongly consistent solutions derived above are illustrated in \Cref{fig:bsp_julia_nonextreme_all}, {and the corresponding outcome vectors are shown by blue circles in Figure~\ref{fig:bsp_julia_outcome}}.
\end{example}

    \begin{rem}\label{rem:selectionstrategy}
    \Cref{alg:DP} may fail to find all strongly vertex consistent solutions if the next node $v_i$ is chosen arbitrarily. In \Cref{ex:missing_nonextreme}, selecting $v_4$ as the initial node prevents the construction of the preference matrix \eqref{eq:matrix_Aup} due to the update \eqref{eq:beta} and thus, the solution with outcome vector $(12,6)^\top$ cannot be derived.
    Therefore, we recommend selecting the node that is lexicographically closest to the facility location among those not yet added nodes as a selection strategy.
    \end{rem}

    Note that in \Cref{ex:missing_nonextreme}, there exist further solutions that are vertex consistent, but not strongly vertex consistent. The {corresponding outcome vectors} are illustrated {by green circles} in \Cref{fig:bsp_julia_outcome}. Even though all these solutions have in common that they express a tied preference, it is in general not sufficient to allow $\beta\cdot\gamma=1$ {to identify them}. Due to the design of \Cref{alg:update}, the existence of at least three path alternatives is necessary to update both parameters $\beta$ and $\g$ in the same iteration. If only one of the parameters is updated, a tied preference cannot be obtained with that specific parameter, since the respective other path will always be discarded.

\subsection{Consistent, but not Strongly Consistent Solutions}\label{app:nonex}

To ensure that we obtain {a minimum complete set for} all consistent, i.\,e., non-extreme supported solutions, we extend the algorithms by a second update function covering the tied preference case. {Towards this end}, we introduce an $n$-tuple $\M=\{\M_1,\dots,\M_n\}$ that {keeps track of} the constructed {preference} matrices in each iteration to reduce the computational effort. The tuple contains a set of matrices $\M_i$ for each node $v_i\in V$, which is initialized in \Cref{alg:DP_nonex} with the empty set for $v_i\in V\setminus\{\facloc\}$ and with $A_{}^{{0}}$ for the facility location $\facloc$. {Note that \Cref{alg:DP_nonex} determines the set of vertex consistent solutions $\VCon$ {by recursively calling} \Cref{alg:CSP_nonex}, similar to Algorithms~\ref{alg:DP} and \ref{alg:CSP} above.}
The adaptions of \Cref{alg:DP_nonex,alg:CSP_nonex}, compared to \Cref{alg:DP,alg:CSP}, are highlighted.

\begin{algorithm}[h!]
\caption{Extended Node-based Algorithm}\label{alg:DP_nonex}
    \KwIn{graph $G=(V,E)$ with $|V|=n$, objective function $c:\X\rightarrow \R^2_{\geq}$, initial preference matrix $A_{}^{{0}}\in\A_2$}
    \KwOut{{minimum complete} set of consistent solutions $\Con$ }
    $S_v\coloneqq\emptyset\quad \forall\:v\in V$\;
    \Forall{$\facloc\in V$}{
    Compute the sets of shortest paths $\ivPaths$ from $\facloc$ to $v_i$ for all $v_i\in V\setminus\{\facloc\}$ with objective function $A^0\cdot c{(\cdot)}$\;
    $\allPaths\coloneqq(\mathcal{P}_1(v),\dots,\mathcal{P}_n(v))$\;
    \highlight{$\M \coloneqq (\emptyset,\emptyset,\dots,\emptyset)$ \;
    $\M_v\coloneqq \{A^0\}$\;
    $\VCon\coloneqq \text{CSP}(G,c,\facloc,\allPaths,A^0,(\facloc,\emptyset),\emptyset,\M)$ \;}
    \Forall{$\sol\in\VCon$}{
    $f(z)\coloneqq A_{}^{{0}}\cdot c(\selectedPaths)$\;
    $S_v\coloneqq S_v\cup(f(z),z)$\;
}
    }
    Filter the set $\Con\coloneqq \bigcup_{\facloc\in V}S_v$ wrt.\ Pareto dominance\;
    \Return $\Con$
\end{algorithm} \begin{algorithm}[ht!]
\caption{Consistent Shortest Paths\\ \hspace*{6.8em}$\text{CSP}(G,c,\facloc,\allPaths,A,G_c,\VCon,\M)$}\label{alg:CSP_nonex}
    \KwIn{graph $G=(V,E)$, objective function $c:\X\rightarrow \R^2_{\geq}$, facility location $\facloc\in V$, $n$-tuple $\allPaths$ containing {Pareto optimal} paths, preference matrix $A\in\A_2$, {graph representing already selected paths $G_c=\parS$,} {partial} set of vertex consistent solutions $\VCon$ for the facility location $\facloc$, \highlight{$n$-tuple $\M$ containing sets of matrices $\M_i$ for all nodes $v_i\in V$}}
    \KwOut{{minimum complete} set of vertex consistent solutions $\VCon$ for the facility location $\facloc$}
Choose $v_i\in V\setminus\Vc$ using a suitable selection strategy\;
        Extract the $v$-$v_i$-paths $\ivPaths$ from $\allPaths$ \; 
        Filter $\ivPaths$ for non-dominance wrt.\ $A\cdot c{(\cdot)}$ \;
        \Forall{$\Path\in\ivPaths$ with $\Path=(\VP,\EP)$}{
            \highlight{$\M_i\coloneqq \emptyset$\;}
            $\Aup\coloneqq\text{update}(A,\Path,\ivPaths,c)$\;
            \If{\highlight{$\aup_{1,2}\cdot \aup_{2,1}\leq1$}}{
                \If{$\Vc\cup\VP\neq V$}{
                $\highlight{\text{genA}(G,c,\facloc,v_i,\allPaths,\Aup,(\Vc\cup\VP,\Ec\cup\EP),\VCon,\M)}$\;
                }
                \Else{
                    {Extract the set $\selectedPaths$  from $(\Vc\cup\VP,\Ec\cup\EP)$}\;
                    $\VCon\coloneqq\VCon\cup{\{(\facloc,\selectedPaths)\}}$\;
                }
            }    
        }
\Return $\VCon$\;
\end{algorithm}
 \begin{algorithm}[h!]
\caption{Generate Additional Preference Matrices\\
\hspace*{6.8em}$\text{genA}(G,c,\facloc,v_i,\allPaths,A,\Aup,G_c,\VCon,\M)$}\label{alg:CSPrec}
    \KwIn{graph $G=(V,E)$, objective function $c:\X\rightarrow \R^2_{\geq}$, facility location $\facloc\in V$, customer location $v_i\in V\setminus\{\facloc\}$, $n$-tuple $\allPaths$ containing {Pareto optimal} paths, preference matrix $A\in\A_2$, updated preference matrix $\Aup\in\A_2$, 
    {graph representing already selected paths $G_c=\parS$,}
{partial} set of vertex consistent solutions $\VCon$ for the facility location $\facloc$, $n$-tuple $\M$ containing sets of matrices $\M_i$ for all nodes $v_i\in V$}
    \KwOut{{minimum complete} set of vertex consistent solutions $\VCon$ for the facility location $\facloc$}
$\Aup_{\beta}\coloneqq \text{updateEQ}(\Aup,1)$\;
    $\Aup_{\g}\coloneqq \text{updateEQ}(\Aup,2)$\;
    $bool_{\beta}\coloneqq false$\;
    $bool_{\g}\coloneqq false$\;
    \If{$\Aup\neq A\land \Aup\notin\M$}{
        $\M_i\coloneqq \M_i\cup \{\Aup\}$
    }
    \If{$\Aup_{\beta}\notin\M$}{
        $\M_i\coloneqq \M_i\cup \{\Aup_{\beta}\}$\;
        $bool_{\beta}\coloneqq true$\;
    }
    \If{$\Aup_{\g}\notin \M$}{
        $\M_i\coloneqq \M_i\cup \{\Aup_{\g}\}$\;
        $bool_{\g}\coloneqq true$\;
    }
    $\VCon\coloneqq \VCon\cup\text{CSP}(G,c,\facloc,v_i,\allPaths,A,\Aup,G_c,\VCon,\M)$\;
    \If{$\Aup_{\beta}\neq \Aup\land\: bool_{\beta}=true$}{
        $\VCon\coloneqq \VCon\cup\text{CSP}(G,c,\facloc,v_i,\allPaths,A,\Aup_{\beta},G_c,\VCon,\M)$\;
    }
    \ElseIf{$\Aup_{\g}\neq \Aup\land \Aup_{\g}\neq \Aup_{\beta}\land\: bool_{\g}=true$}{
        $\VCon\coloneqq \VCon\cup\text{CSP}(G,c,\facloc,v_i,\allPaths,A,\Aup_{\g},G_c,\VCon,\M)$\;
    }
    \Return $\VCon$\;
\end{algorithm} \begin{algorithm}[h!]
\caption{$\text{updateEQ}(A,q)$}\label{alg:updateEQ}
    \KwIn{preference matrix $A\in\A_2$, indicator $q\in\{1,2\}$}
    \KwOut{updated preference matrix $A\in\A_2$}
    $\beta\coloneqq {a}_{1,2}$\;
    $\g\coloneqq {a}_{2,1}$\;
    \If{$q=1\land\beta\neq\beta_0$}{
        ${a}_{2,1}\coloneqq \frac{1}{\beta}$\;
    }
    \ElseIf{$q=2\land\g\neq\g_0$}{
        ${a}_{1,2}\coloneqq \frac{1}{\g}$\;
    }
    \Return $A$\;
\end{algorithm} 
In \Cref{alg:CSPrec}, we construct two additional preference matrices $\Aup_{\beta}$ and $\Aup_{\gamma}$ expressing the tied preferences $\beta=\frac{1}{\gamma}$ and, vice versa, $\gamma=\frac{1}{\beta}$ using \Cref{alg:updateEQ}. We can skip this step, if the respective parameter $\beta$ or $\g$ has not been updated compared to its initial value \(\beta_0\) or \(\g_0\), respectively.
By the additional construction of these two matrices we ensure that we also cover the tied cases even if the graph lacks a sufficient number of path alternatives for them to be constructed with \Cref{alg:update}. To avoid multiple redundant iterations with the same preference matrices, we add the matrices $\Aup_{\beta}$ and $\Aup_{\gamma}$ to the set $\M_i$ if they have not been added to $\M_i$ in a previous iteration for node $v_i\in V$. If one of the matrices has been added before, we discard the respective matrix $\Aup_{\beta}$ or $\Aup_{\gamma}$. In this case, using these matrices would only lead to the same solutions that we also obtain with the same matrix starting from the first node where the matrix was constructed. In contrast to the matrices $\Aup_{\beta}$ and $\Aup_{\g}$ we always continue to construct a vertex consistent solution with the matrix $\Aup$, even if the matrix has not been updated. In this case, we may not have gained any further preference information by adding path $\Path$, but as long as there exist nodes that have not been added yet, we continue with the construction process. Note that with this strategy, redundant iterations are significantly reduced, but not avoided completely. If the matrices have not been considered in previous iterations, we call \Cref{alg:CSP_nonex} not only with the matrix $\Aup$, but also with the matrices $\Aup_{\beta}$ and $\Aup_{\gamma}$.

To ensure that in the remaining part of the solution all feasible combinations of weights and paths can be selected, we set $\M_i=\emptyset$ after each path choice in \Cref{alg:CSP_nonex}. As soon as there exists more than one path alternative, we have to guarantee that in all partial solutions beginning with node $v_i$ we can make the same decisions. Hence, we discard the matrices we have already constructed in this iteration to enable the construction of the same matrices in the partial solutions in the later construction process.

\section{Numerical Results}\label{sec:numeric}
We test our algorithms with real world test instances generated with \cite{OpenStreetMap}. The test instances are created around a central address with radii between 100 and 250 meters, and we select central addresses in the city centers of Wuppertal, Germany, and Münster, Germany. For details of the  construction of such test instances, we refer to \cite{boeing2025modelinganalyzingurbannetworks}. We model our test instances as ordinal location problems, where the edge lengths correspond to the distance between nodes. We assign category $\cat_1$ (green) if there exists a bicycle {lane} and category $\cat_2$ (red) if not. The resulting graphs are simplified, see \cite{boeing2025topologicalgraphsimplificationsolutions}, and vary in size. We obtain graphs consisting of up to 123 nodes and up to 312 edges. We only consider those nodes as possible facility and customer locations that {correspond to} a real address and denote this set of nodes by $V'\subseteq V$  with $|V'|=n'\leq n$. Nodes representing junctions or other infrastructure are not regarded as feasible facility locations. 

\begin{table}[ht!]
\centering\small
\begin{tabular}{rr @{\qquad} rrr @{\qquad}rr @{\qquad}rr }\toprule
$n$ & $n'$ & $m$ & $\cat_1$ & $\cat_2$ & $|\SCon|$ & $|\VSCon|$ & $|\Con|$ & $|\VCon|$ \\ \midrule
8 & 8 & 18 & 6 & 12 & 4 & 19 & 4 & 39\\
10 & 10 & 18 & 16 & 2 & 1 & 10 & 1 & 10\\
19 & 7 & 46 & 8 & 38 & 1 & 19 & 1 & 88\\
23 & 9 & 62 & 32 & 30 & 9 & 115 & 9 & 831\\
24 & 20 & 60 & 18 & 42 & 9 & 143 & 11 & 700\\
38 & 16 & 94 & 22 & 72 & 4 & 221 & 4 & 12.060\\
43 & 34 & 92 & 64 & 28 & 2 & 141 & 3 & 957\\
50 & 34 & 120 & 24 & 96 & 10 & 194 & 12 & 1.805\\ 
52 & 22 & 130 & 32 & 98 & 5 & 449 & 10 & 70.646\\
69 & 28 & 192 & 64 & 128 & 32 & 2.498 & 41 & 20.247\\
79 & 27 & 210 & 48 & 162 & 16 & 727 & 21 & 136.587\\
82 & 47 & 216 & 72 & 144 & 31 & 1.574 & 42 & 35.578\\
96 & 41 & 222 & 108 & 114 & 16 & 734 & 22 & 6.453\\
102 & 26 & 268 & 82 & 186 & 41 & 1.537 & 45 & 17.374\\
123 & 7 & 312 & 78 & 234 & 47 & 486 & 60 & 8.592 \\
\bottomrule
\end{tabular}
\caption{15 test instances with $n$ nodes, $n'$ {customer (and candidate)} locations, and $m$ edges with the respective ratio of edges in category $\cat_1$ and $\cat_2$, the number of strongly consistent solutions $|\SCon|$, the number of strongly vertex consistent {outcome vectors} $|\VSCon|$, the number of consistent {outcome vectors} $|\Con|$, and the number of vertex consistent {outcome vectors} $|\VCon|$.}
\label{table:test_instances}
\end{table}

In \Cref{alg:DP}, we now iterate over the subset $V'$ rather than the complete set $V$. The graphs also vary in their ratio of green and red edges, where we observe that most of our test instances have a higher proportion of red edges while only four instances have more green or a level number of green and red edges. In \Cref{table:test_instances}, we evaluate the number of strongly consistent {outcome vectors} $|\SCon|$, the number of strongly vertex consistent {outcome vectors} $|\VSCon|$, the number of consistent {outcome vectors} $|\Con|$ and the number of vertex consistent {outcome vectors} $|\VCon|$. We observe that for small instances, the number of consistent and strongly consistent {outcome vectors} matches, while for larger instances, the number of consistent {outcome vectors} that are not strongly consistent increases significantly. This difference is larger when comparing the number of vertex consistent {outcome vectors} and the number of strongly vertex consistent {outcome vectors}. Since we iteratively construct {a minimum complete set of} (strongly) vertex consistent solutions and filter these sets for {globally, i.e., over all vertices,} non-dominated solutions afterwards, the computational effort of computing all consistent {outcome vectors} is significantly higher. While the runtime to determine the set of strongly consistent {outcome vectors} is less than 75 seconds for all instances, it takes up to 40 minutes to compute all consistent {outcome vectors}. \Cref{fig:bsp_solutions} illustrates three strongly consistent solutions and their respective efficient facility location.

\begin{figure}[ht!]
    \centering
    \begin{minipage}[b]{0.32\textwidth}
        \includegraphics[width=\textwidth, keepaspectratio]{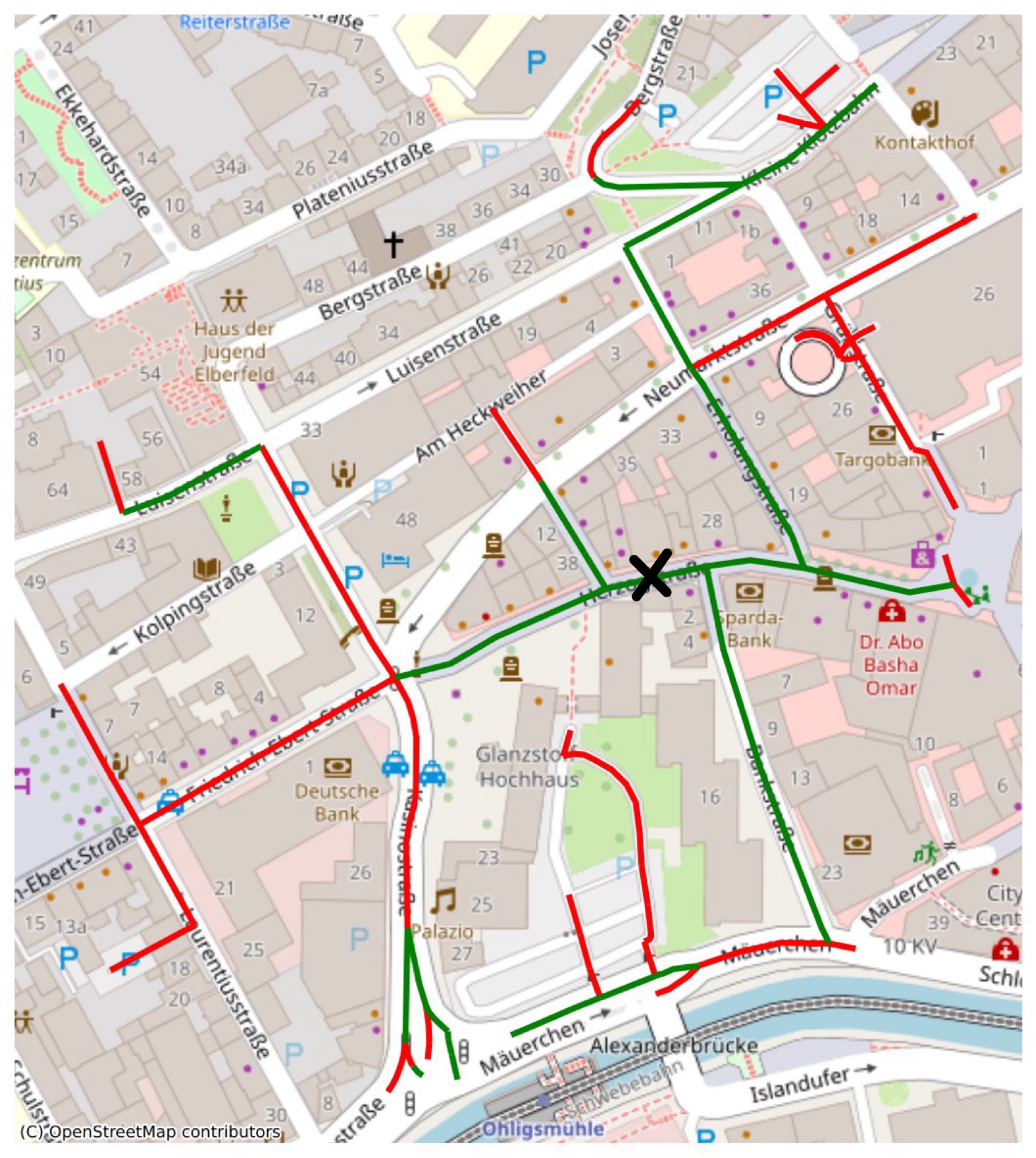}
    \end{minipage}
    \hfill
    \begin{minipage}[b]{0.32\textwidth}
        \includegraphics[width=\textwidth, keepaspectratio]{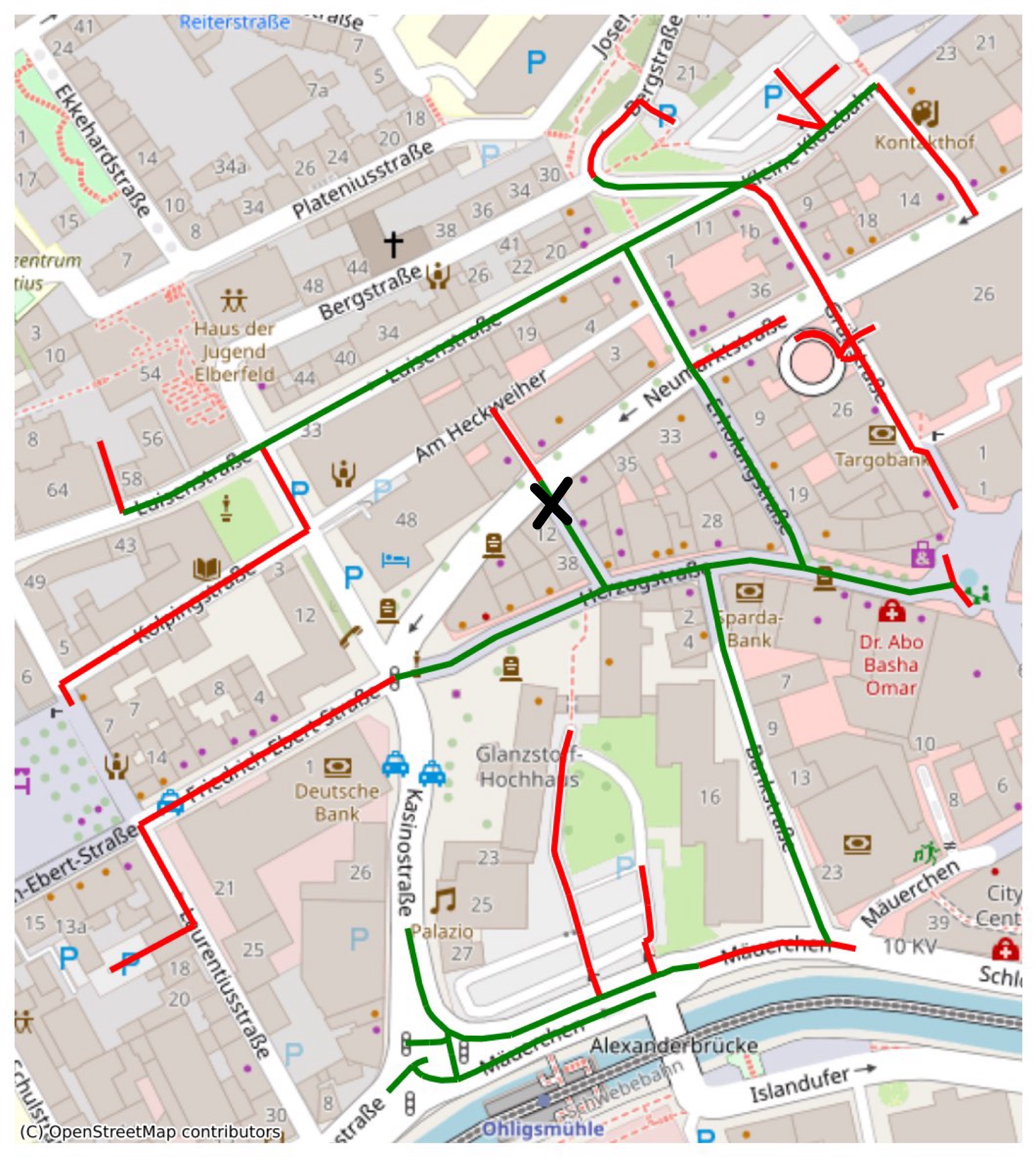}
    \end{minipage}
    \hfill
    \begin{minipage}[b]{0.32\textwidth}
        \includegraphics[width=\textwidth, keepaspectratio]{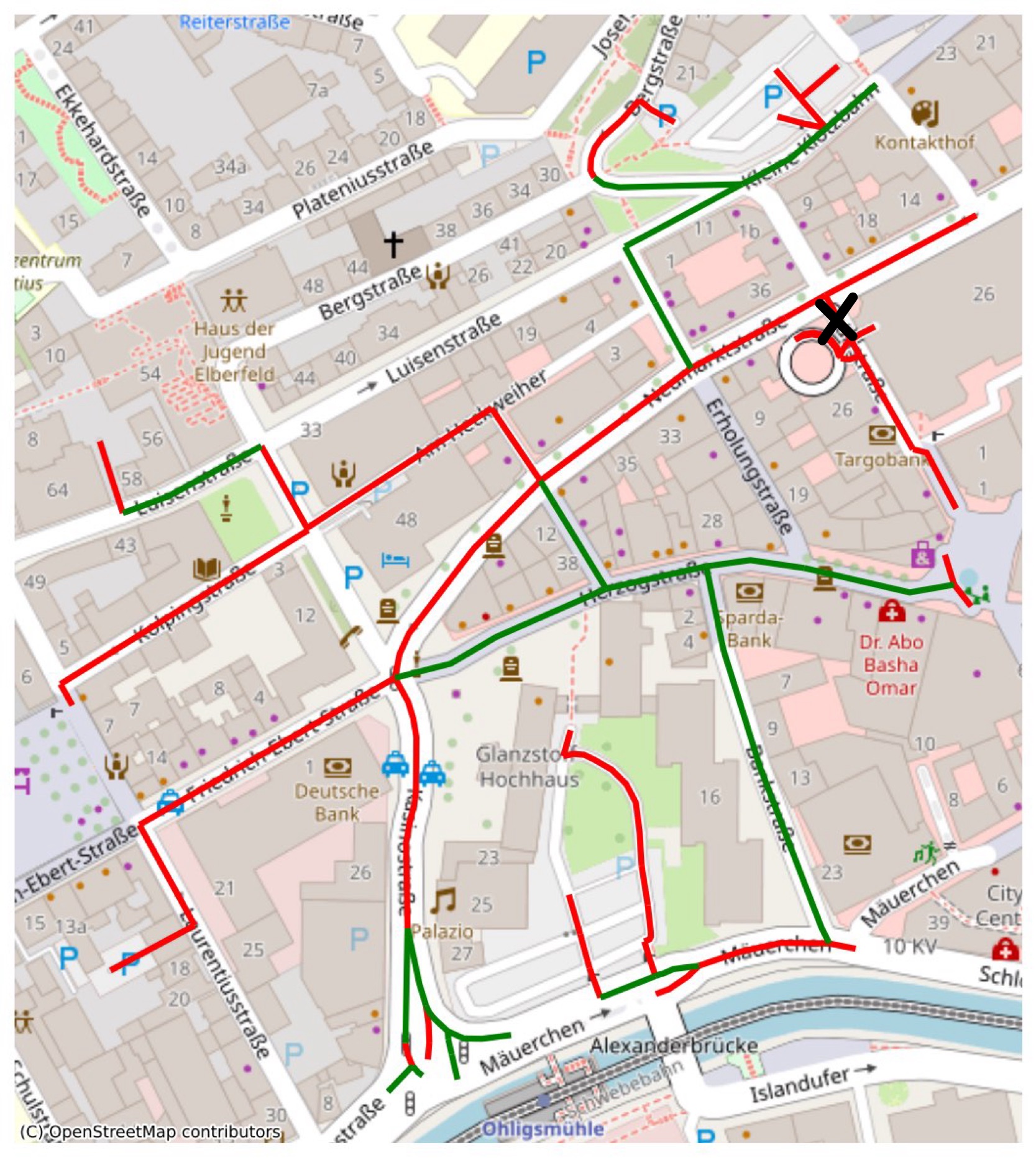}
    \end{minipage}

    \caption{Three strongly consistent solutions and their Pareto optimal facility locations in a radius of 200m around Herzogstr. 40, 42103 Wuppertal. Edges in category $\cat_1$ are depicted in green, edges in category $\cat_2$ in red, the Pareto optimal facility location is marked with a black cross.}
    \label{fig:bsp_solutions}
\end{figure}

\section{Conclusion}\label{sec:concl}
In this paper, we investigate bi-objective median location {and path planning} problems with {bi-}objective edge costs. We introduce the concept of a (strongly) consistent path choice and relate it to (extreme) supportedness. We prove that it is sufficient to consider the set of nodes as the set of feasible locations, since each non-dominated outcome vector has at least one pre-image located in a node. Furthermore, we show that for graphs with unique paths to all nodes, like line graphs or trees, the (Pareto) optimal solutions only depend on the {demands}, as long as the edge costs are non-negative and non-zero. {Hence, in this situation we can generalize our results to an arbitrary number of objective functions.}

For general graphs, we present an algorithm that computes {a minimum complete set of} strongly consistent solutions for bi-objective location problems. Moreover, we adapt this algorithm such that {a minimum complete set of} consistent solutions is computed. The concept of consistent paths as well as the algorithms are illustrated at an example of an ordinal location problem with two categories. We solve some ordinal location problems with two categories on {several real-world} instances which were derived from OpenStreetMap data. With growing instance size, we observe that the number of consistent but not strongly consistent {outcome vectors} increases significantly. Future work could investigate median location problems with consistent path choice with more than two objectives.

\subsection*{Acknowledgements}\noindent
This work is partially supported by the  project SAFeR: Safe, Algorithm-based Footpath Routing funded by the EU program EFRE/JTF NeueWege.IN.NRW (EFRE 20800941).

\end{document}